\documentclass{amsart}

\usepackage{amsfonts}
\usepackage{amsmath}
\usepackage{amssymb}
\usepackage{amsthm}
\usepackage{caption}
\usepackage{enumitem}
\usepackage[margin=1.1in]{geometry}
\usepackage{graphicx}
\usepackage{hyperref}
\usepackage{mathrsfs}
\usepackage{palatino}
\usepackage{subcaption}
\usepackage{svg}
\usepackage{tabularray}
\usepackage{tikz}
\usepackage{tikz-cd}
\usepackage[normalem]{ulem}
\usepackage{xcolor}

\hypersetup{
    colorlinks=true,
    linkcolor=blue,
    citecolor=blue,
    filecolor=magenta,      
    urlcolor=cyan,
    pdftitle={Persistent Legendrian Contact Homology in $\mathbb{R}^3$},
    pdfpagemode=FullScreen,
    }

\usetikzlibrary{arrows.meta}
\tikzset{    
    mypoint/.style={
        circle,
        draw,
        inner sep=.3mm
        },  
    whitepoint/.style={
        fill=white, 
        mypoint
        },  
    blackpoint/.style={
        fill=black, 
        mypoint
        },  
    redpoint/.style={
        draw=red,
        fill=red, 
        mypoint
        },  
    redwhitepoint/.style={
        draw=red,
        fill=white, 
        mypoint
        },  
    textnode/.style={
        text height=2.5ex, 
        text depth=1ex
        },  
    }

\newcommand{\Alg}{\mathcal{A}}
\newcommand{\AlgL}{{\mathcal{A}_\Lambda}}
\newcommand{\dL}{\partial_\Lambda}
\newcommand{\deps}{\partial^\epsilon}
\newcommand{\id}{\mathrm{id}}
\newcommand{\Leg}{\Lambda}
\newcommand{\std}{\mathrm{std}}

\DeclareMathOperator{\im}{im}
\DeclareMathOperator{\LCH}{LCH}
\DeclareMathOperator{\nul}{null}
\DeclareMathOperator{\rank}{rank}
\DeclareMathOperator{\rot}{rot}

\newcommand{\R}{\mathbb{R}}
\newcommand{\Z}{\mathbb{Z}}

\newenvironment{subsubfigure}[2][]{%
  \begin{subfigure}[#1]{#2}%
    \stepcounter{subsubfigure}%
}{%
    \addtocounter{subfigure}{-1}%
  \end{subfigure}%
}
\newcounter{subsubfigure}

\newtheorem{theorem}{Theorem}[section]
\newtheorem{proposition}[theorem]{Proposition}
\newtheorem{lemma}[theorem]{Lemma}

\theoremstyle{definition}
\newtheorem{definition}{Definition}

\newtheorem{example}[theorem]{Example}
\newtheorem{remark}[theorem]{Remark}
\newtheorem{question}{Question}[section]

\begin{document}

\title[Persistent Legendrian Contact Homology in $\R^3$]{Persistent Legendrian Contact Homology in $\R^3$}

\author[Basu, Christian, Clayton, Irvine, Mooers, and Shen]{
	Maya Basu
	\and
	Austin Christian
	\and
	Ethan Clayton
	\and
	Daniel Irvine
	\and
	Fredrick Mooers
	\and
	Weizhe Shen
}

\maketitle

\begin{abstract}
This work applies the ideas of persistent homology to the problem of distinguishing Legendrian knots. We develop a persistent version of Legendrian contact homology by filtering the Chekanov-Eliashberg DGA using the action (height) functional. We present an algorithm for assigning heights to a Lagrangian diagram of a Legendrian knot, and we explain how each Legendrian Reidemeister move changes the height of generators of the DGA in a way that is predictable on the level of homology. More precisely, a Reidemeister move that changes an area patch of a Lagrangian diagram by $\delta$ will induce a $2\delta$-interleaving on the persistent Legendrian contact homology, computed before and after the Reidemeister move. Finally, we develop strong Morse inequalities for our persistent Legendrian contact homology.
\end{abstract}

\section{Introduction}\label{sec:intro}

This work presents a new lens through which we view the problem of classifying Legendrian knots.

A Legendrian knot is a smooth embedding $S^1\to \mathbb{R}^3$ that is everywhere tangent to the contact distribution given by the standard contact form, 
\[ \alpha_{std} = dz - y\, dx.\]
More precisely, if a knot $\Lambda$ is parametrized as $\Lambda(t) = (x(t),y(t),z(t))$ for $t\in[0,2\pi]$, then the Legendrian condition asks that
\begin{equation}\label{eq:LegCondition} z'(t) - y(t)\, x'(t) = 0.\end{equation}
By imposing this geometrical constraint on a knot, we are led to consider a rich problem of classifying knots up to isotopies which preserve their topological and geometrical properties. This, broadly, is the classification problem in Legendrian knot theory.

A Legendrian knot properly exists in $\mathbb{R}^3$, but it is often convenient to work with a two-dimensional projection of the knot. There are two useful ways of projecting a Legendrian knot into the plane. The Lagrangian projection is 
\[ \Pi \colon \mathbb{R}^3\to \mathbb{R}^2, \qquad \text{ given by } \qquad (x,y,z) \mapsto (x,y). \]
The Lagrangian projection operator, when restricted to a Legendrian knot $\Lambda$, gives a particular \textit{Lagrangian diagram} of $\Lambda$. More precisely, a Lagrangian diagram is a knot diagram that is a Lagrangian projection of a Legendrian knot, modulo all possible planar isotopies.
(See Section~\ref{sec:legs}.) An example Lagrangian diagram is shown in Figure~\ref{fig:toyunknot(a)}. The other important projection is the front projection operator.
\[ F \colon \mathbb{R}^3\to \mathbb{R}^2, \qquad \text{ given by } \qquad (x,y,z) \mapsto (x,z). \]

Legendrian contact homology (abbreviated LCH) is a powerful tool to use on the classification problem. The survey \cite{etnyre2020legendrian} explains how to compute the Legendrian contact homology, $\LCH_*(\Lambda)$, given a Lagrangian diagram for $\Lambda$.  In this introduction, we elide the details of how to use \emph{augmentations} of $\LCH_*(\Lambda)$ to extract useful information, but an explanation is given in Section~\ref{linearize}.

  The main utility of LCH comes in distinguishing Legendrian knots. We describe here an example application, in the broadest possible terms. Given two Legendrian knots $\Lambda_1$ and $\Lambda_2$, one could compute the so-called Poincar\'e-Chekanov polynomials $PC^{\epsilon}$ for $\Lambda_1$ and $\Lambda_2$  with respect to all possible augmentations of each (see Theorem~\ref{thm:linearized-invariant} and Formula~\ref{PCpoly}). If one arrives at two distinct collections of polynomials, then one can conclude that $\Lambda_1$ and $\Lambda_2$ are distinct as Legendrian knots. 

The substance of this work is to introduce ideas of persistent homology (abbreviated persistence) into the computation of the Legendrian contact homology of a knot. True to its name, persistent homology measures the geometro-topological features that persist throughout the many spaces that make up a filtered knot. Given a knot $\Lambda$, a filtration is an increasing family of subsets, $\{\Lambda^a\}_{a\in \mathbb{R}}$, such that $\Lambda^a \subseteq \Lambda^b$, whenever $a \leq b$, and $\Lambda^a=\Lambda$ for sufficiently large $a$. While this definition is useful for intuition, in this paper we will not examine a filtration on a Legendrian knot; instead, we will define a filtration on the chain level. A persistence module of chain complexes is a functor from the partially-ordered set of real numbers, $(\mathbb{R},\leq)$ treated as a poset category, to the category of chain complexes of modules over a fixed ground ring. A more detailed definition will be given in Section~\ref{background:pmods}.

Legendrian contact homology is the homology of a differential graded algebra referred to in the literature as the \emph{Chekanov-Eliashberg DGA}.  The Chekanov-Eliashberg DGA is naturally filtered by the Reeb action of each generator. This will be the filtration for our persistence module structure. We adopt the following sign convention. Each generator $q$ of the Chekanov-Eliashberg DGA has a geometric realization as a crossing in a Lagrangian projection. Since the Legendrian knot $\Lambda$ exists in contact $\mathbb{R}^3$, we may lift the crossing from the Lagrangian projection $\Pi(\Lambda)$ to a pair of points $\tilde{c}\in \Lambda$ and define
\[ \text{action}(q) = \int_{\tilde{c}} dz. \]
A chain level filtration will turn the Chekanov-Eliashberg DGA into a persistence module with a differential operator, as defined in \cite{polterovic2017persistence}. The Chekanov-Eliashberg DGA already has a $\mathbb{Z}$-grading, called the Maslov grading, given by the degree of a word. Introducing a filtration adds an $\mathbb{R}$-grading. As a means of distinguishing these two gradings, we call $\text{action}(q)$ the \textit{height} of the element $q$. Henceforth we denote the height by $h(q)$. The evaluation of heights agrees with the algebra operations of the Chekanov-Eliashberg DGA;
\[ h(q_1q_2) = h(q_1)+h(q_2)  \qquad \text{ and } \qquad h(q_1 + q_2) = \max\{h(q_1),h(q_2)\}.\]

We now see that the problem of introducing persistence into Legendrian contact homology hinges on the assignment of a height to each generator of the DGA. In this paper we will explain why it is natural to express the LCH of a Legendrian knot as a persistence module and demonstrate a powerful tool for visualizing persistent LCH, called a barcode.

\begin{lemma}\label{lem:decomp}
    In this work, we are interested in the following two algebraic structures, both of which are persistence modules.
    \begin{enumerate}
        \item The Chekanov-Eliashberg DGA with a height filtration is an example of a persistence module of DGAs.
        \item The (linearized) homology of (i) is a persistence module with a real parameter, denoted $F^\bullet \LCH^\epsilon_\ast(\Leg)$.
    \end{enumerate}
\end{lemma}

All persistence modules of interest to us can be decomposed in the following way.

\begin{theorem}\label{thm:decomposition}
\emph{(A consequence of the structure theorem of Gabriel, Auslander, Ringel– Tachikawa, Webb, cf. \cite{chazal2016strcture})} The persistence modules in Lemma~\ref{lem:decomp} can be decomposed into a direct sum of interval modules.
\end{theorem}

An interval module is the simplest example of a persistence module, and this is explained in Example~\ref{ex:interval}, below. We use this structure theorem every time that we compute persistent homology. 

There is a pseudo-metric on the class (category) of persistence modules of Theorem~\ref{thm:decomposition}, called the \emph{interleaving distance}. The interleaving distance will be defined in Section~\ref{sec:nearby}, but for now, let us denote the distance function by
\[ d\colon(\text{persistence modules})^{\times 2} \to \mathbb{R}. \]

The ultimate goal of this investigation is to explain the extent to which the persistent Legendrian contact homology is an invariant of a Lagrangian projection diagram of a Legendrian knot $\Lambda$. In order to do this, we must show that applying any reasonable Legendrian Reidemeister move to a Lagrangian diagram for $\Lambda$ will change the persistent LCH by a small amount, with respect to this interleaving distance.

\begin{theorem}\label{thm:main-invariance-result}
Let $\Lambda_0$ and $\Lambda_1$ be Legendrian knots whose Lagrangian projections $\Pi(\Lambda_0)$ and $\Pi(\Lambda_1)$ are related by a single Legendrian Reidemeister move, and fix $\delta>0$.  There exist Legendrian knots $\tilde{\Lambda_0}$ and $\tilde{\Lambda_1}$ whose Lagrangian projections are planar isotopic to those of $\Lambda_0$ and $\Lambda_1$, respectively, such that
\[
d(F^\bullet \LCH^\epsilon_\ast(\tilde{\Lambda_0})\, , \, F^\bullet \LCH^\epsilon_\ast(\tilde{\Lambda_1})) < \delta.
\]
\end{theorem}

Let us discuss a bit of motivation for Theorem~\ref{thm:main-invariance-result}. Recall that $F^\bullet \LCH^\epsilon_\ast(\Leg)$ is a persistence module of a real parameter. After applying Theorem~\ref{thm:decomposition}, we write $F^\bullet \LCH^\epsilon_\ast(\Leg)$ as a direct sum of interval modules, which we call a barcode. The interleaving distance measures the infimum over all $\delta$ for which a $2\delta$ interleaving exists between two barcodes.  A working definition for a $2\delta$-interleaving of a barcode is as follows. One is allowed to move the endpoints of any bar left or right by a distance of at most $\delta$. Consequently, one may delete any bar of finite length less than $2\delta$. 

Using this working definition, one might draw a barcode of a Legendrian knot diagram and look for a bar (interval) of length $<2\delta$, for some pre-determined noise parameter $\delta$. By identifying the crossings that generate the endpoints of the bar of short length, one may then look for a Legendrian Reidemeister move to perform near those crossings which can simplify the knot and the barcode. This is a way in which a barcode could hope to distinguish essential information from noise.

\subsection*{Future directions}
The present article establishes the definition of persistent Legendrian contact homology and offers some basic tools for computing persistent LCH in practice.  But several interesting questions remain, and we plan to address some of these in future work.  Our hope is that persistent LCH will have much to tell us about the classification of Legendrian knots in $\mathbb{R}^3$.

Our first and most fundamental question is one of invariance.  While a fixed (generic) Legendrian knot in $\mathbb{R}^3$ has a well-defined persistent LCH, Legendrian knots are typically studied according to their Legendrian isotopy classes.  The invariance question thus becomes
\begin{question}
What collection of persistence modules can be realized as the persistent LCH of a fixed Legendrian isotopy type?
\end{question}
Because Legendrian contact homology is computed combinatorially from a Lagrangian diagram, we find it natural to decompose this question into two finer questions:
\begin{enumerate}
    \item How is persistent LCH affected by Legendrian isotopies which do not change the Lagrangian diagram?\label{q:planar}
    \item How is persistent LCH affected by Legendrian isotopies which induce Reidemeister moves on the Lagrangian diagram?\label{q:reidemeister}
\end{enumerate}
Theorem~\ref{thm:main-invariance-result} provides one answer to~\ref{q:reidemeister}: a Legendrian isotopy $\Lambda_t$, $t\in[t_0-\epsilon,t_0+\epsilon]\subsetneq[0,1]$, corresponding to a single Lagrangian Reidemeister move at $t=t_0$ may be extended to a Legendrian isotopy $\Lambda_t$, $t\in[0,1]$, whose Lagrangian projection is constant on each of $[0,t_0)$ and $(t_0,1]$ and whose endpoints $\Lambda_0,\Lambda_1$ have arbitrarily small interleaving distance between their persistent Legendrian contact homologies.

Our understanding of~\ref{q:planar} is also quite strong.  Namely, a planar isotopy of a Lagrangian projection cannot change LCH, and thus persistent LCH will change only by variations in the filtration --- that is to say, the endpoints of a bar in the persistent LCH barcode may change, but the number of bars in the barcode will not.  For a fixed Lagrangian diagram with $n$ crossings, the collection of valid filtrations forms an open cone in the first orthant of $\mathbb{R}^n$, and Legendrian isotopies which leave the Lagrangian diagram unchanged correspond to paths in this cone.  It is desirable to have a preferred height assignment within this cone, so that persistent LCH is a well-defined invariant of Lagrangian diagrams.  Note that we cannot choose the vertex of our cone, since this assigns a height of zero to every crossing.  The lack of obvious choice for a canonical height assignment leaves us with the following question:
\begin{question}
How should we upgrade persistent LCH to a planar isotopy invariant of Lagrangian diagrams?
\end{question}
Before leaving this question, we outline a strategy which the authors find promising.  Each crossing of a Lagrangian diagram $D$ corresponds to a coordinate $h_i$ in $\mathbb{R}^n$, and this coordinate must satsify $h_i>0$.  The Legendrian condition then leads to $n+1$ inequalities in these coordinates, each of which has a constant term of 0.  To obtain a closed set of feasible height assignments in $\mathbb{R}^n$, we may fix parameters $0<\epsilon_A\ll\epsilon_H$ and require that $h_i\geq\epsilon_H$, while the constant term of each area inequality coming from $D$ is replaced by $\epsilon_A$.  The solutions to the resulting collection of inequalities form a closed, convex, unbounded subset of the cone of height assignments described above, with a unique vertex minimizing the sum $h_1+\cdots+h_n$.  We suspect that defining the persistent LCH of $D$ to be computed using the height assignment corresponding to this vertex will allow us to define a Legendrian isotopy invariant with desirable features, and these suspicions are the subject of ongoing investigation.

In addition to our goal of realizing persistent LCH as a Legendrian isotopy invariant, we hope to replicate some of the additional structure enjoyed by LCH in the persistent setting.  For instance, Sabloff constructed in \cite{sabloff2006duality} an expanded Chekanov-Eliashberg DGA from which one may derive a type of Poincar\'{e} duality for Legendrian contact homology.  In persistent LCH, this gives a duality for infinite bars --- if $r_k\geq 0$ is the number of infinite bars in grading $k$ for a persistent LCH barcode, then $r_1=r_{-1}+1$ and $r_k=r_{-k}$ for $k\neq\pm 1$.  It is natural to wonder how our filtration of the Chekanov-Eliashberg DGA might extend to the expanded algebra, and whether this might lead us to a Poincar\'{e} duality for persistent LCH.  A concrete version of this question is as follows.
\begin{question}
Does persistent LCH admit a notion of Sabloff duality for finite bars?
\end{question}

Our optimism for persistent LCH lies in the fact that it retains the computability of linearized LCH while also capturing chain-level data in the form of finite bars.  However, some nonlinear aspects of LCH can be computed, and have proven to be quite useful: in \cite{civan2011product}, cup and Massey products are constructed on linearized LCH and used to recover higher-order linearizations of LCH.  The authors find it quite plausible that these innovations could be duplicated in the persistent setting, and thus pose the following question.

\begin{question}
Does persistent Legendrian contact homology admit a Massey product structure?
\end{question}

Since its introduction, Legendrian contact homology has been utilized and generalized in a great variety of manners, and we are hopeful that some of these successes will be replicated in persistent Legendrian contact homology.

\subsection*{Organization}
This paper is organized as follows. In Section~\ref{sec:background}, we provide relevant background information on persistent homology, barcodes, the interleaving distance, Legendrian knots, Legendrian contact homology, and augmentations. We then describe how these ideas will fit together in this work to accomplish the goal of using persistent homology to understand Legendrian knots. In Section~\ref{sec:toys}, we provide an outline of the procedure for computing the Legendrian contact homology barcode from a Lagrangian diagram. Two worked examples are given, of the standard unknot and the standard trefoil knot.

In Section~\ref{sec:proof}, we set about proving the main theorem of this work, which is Theorem~\ref{thm:main-invariance-result} above. We need to show that any Legendrian Reidemeister move performed on a Lagrangian diagram will perturb in an inappreciable way the area that is contained by the bounded regions of the diagram. This perturbation of areas will, in turn, perturb the heights that can and must be assigned to each Lagrangian diagram double point. These double points will serve as the chain complex generators of Legendrian contact homology, and the height of each generator will be important for the filtration that gives us persistent homology, as explained in Section~\ref{sec:background}. Hence, to complete the proof of the main theorem we must also check that each Legendrian Reidemeister move perturbs generators of the aforementioned filtered chain complex in a way that will be predictable on the level of homology. 

In Section~\ref{sec:flooding}, we present our \textit{flooding algorithm} for assigning heights to a Lagrangian diagram. In Section~\ref{sec:morse-inequalities}, we associate to a Legendrian knot two polynomials, called the Morse-Chekanov polynomial and the Poincar\'e-Chekanov polynomial. We develop the so-called strong Morse inequalities in Theorem~\ref{thm:strong-morse-inequality}, which says that the Morse-Chekanov polynomial majorizes the Poincar\'e-Chekanov polynomial.

\section{Background}\label{sec:background}
In this section, we gather background material for persistent homology in Section~\ref{background:persistent} and for Legendrian contact homology in Section~\ref{background:lch}.

\subsection{Persistent homology}\label{background:persistent}

Persistent homology is a popular tool in topological data analysis.  In this section, we recall the notation and results necessary to build a \emph{barcode} for a filtered DGA, and describe a pseudo-metric on the category of persistence modules.

\subsubsection{Persistence modules}\label{background:pmods}

\begin{definition}
Let $R$ be a fixed ground ring.  A \emph{persistence module over a real parameter} is a functor from the poset category $(\mathbb{R},\leq)$ to the category of $R$-modules.  Similarly, a \emph{persistence module of DGAs} is a functor from $(\mathbb{R},\leq)$ to the category of DGAs over $R$.
\end{definition}

In practice, our ground ring $R$ will always be $\mathbb{Z}_2$, and the unmodified phrase \emph{persistence module} will refer to a persistence module over a real parameter.  We will be interested in persistence modules of DGAs when working at the chain level, and persistence modules over $\mathbb{Z}_2$ once we have passed to homology.  The data we will track at the chain level consists of an $\mathbb{R}$-graded direct sum
\[
\mathcal{A}^\bullet:= \bigoplus_{t\in\mathbb{R}}\mathcal{A}^t
\]
of DGAs $\mathcal{A}^t$ over $\mathbb{Z}_2$, along with a DGA morphism $\psi_s^t\colon\mathcal{A}^s\to\mathcal{A}^t$, for every pair of real numbers $s\leq t$, satisfying
\begin{enumerate}
    \item $\psi_t^t=\id$, for every $t\in\mathbb{R}$;
    \item $\psi_s^t\circ \psi_r^s = \psi_r^t$, for every triple $r\leq s\leq t$.
\end{enumerate}
We refer to $\psi_s^t$ as a \emph{transfer map}.  At the homology level we will consider an $\mathbb{R}$-graded direct sum of $\mathbb{Z}_2$-vector spaces $V^t$ along with linear transfer maps $v_s^t\colon V^s\to V^t$ satisfying conditions analogous to those above.  We will use the notation $\mathcal{A}^\bullet$ for a persistence module of DGAs and $V^\bullet$ for a persistence module over a real parameter.

\begin{example}\label{ex:interval}
For our purposes, the fundamental example of a persistence module is the \emph{interval module}.  Given any interval $I\subset\mathbb{R}$, we may define a family of $\mathbb{Z}_2$-vector spaces
\[
V^t := \left\lbrace\begin{matrix}
    \mathbb{Z}_2, & t\in I\\
    0, & \text{otherwise}
\end{matrix}\right.,
\]
along with linear transformations $v_s^t\colon V^s\to V^t$ defined for $s\leq t$ by
\[
v_s^t := \left\lbrace\begin{matrix}
    \mathrm{id}, & s,t\in I\\
    0, & \text{otherwise}
\end{matrix}\right..
\]
This collection of vector spaces and linear maps determines a persistence module $V^\bullet$; if $I=[a,b)$, we will denote this persistence module by $\mathbb{Z}_2[a,b)$.  

We will see that persistent homology is expressed as a direct sum of interval modules, and this decomposition is called a barcode.
\end{example}

\subsubsection{Persistent homology}
The homology of a DGA $(\mathcal{A},\partial)$ over $\mathbb{Z}_2$ is a direct sum of vector spaces:
\[
H_*(\mathcal{A},\partial) := \bigoplus_{k\in\mathbb{Z}}\dfrac{\ker\partial_k}{\im\partial_{k+1}}.
\]
Analogously, the persistent homology of a persistence module of DGAs $(\mathcal{A}^\bullet,\partial^\bullet)$ over $\mathbb{Z}_2$ is a $\mathbb{Z}$-graded direct sum of persistence modules.  For each integer $k\in\mathbb{Z}$, we have the persistence module 
\[
H_k(\mathcal{A}^\bullet,\partial^\bullet) := \bigoplus_{t\in\mathbb{R}} H_k(\mathcal{A}^t,\partial^t),
\]
with transfer maps $H_k(\mathcal{A}^s,\partial^s)\to H_k(\mathcal{A}^t,\partial^t)$ induced by the chain-level transfer maps $\psi_s^t\colon\mathcal{A}^s\to\mathcal{A}^t$.  From these we obtain the persistent homology of $(\mathcal{A}^\bullet,\partial^\bullet)$:
\[
H_*(\mathcal{A}^\bullet,\partial^\bullet) := \bigoplus_{k\in\mathbb{Z}} H_k(\mathcal{A}^\bullet,\partial^\bullet).
\]
Thus, for instance, the vector space of $H_*(\mathcal{A}^\bullet,\partial^\bullet)$ at filtration level $t\in\mathbb{R}$ is given by
\[
H^t_*(\mathcal{A}^\bullet,\partial^\bullet) = \bigoplus_{k\in\mathbb{Z}} H_k(\mathcal{A}^t,\partial^t),
\]
and the transfer map between the vector spaces at filtration levels $s,t$, with $s\leq t$, is given as a direct sum of transfer maps.

\subsubsection{Nearby persistence modules}\label{sec:nearby}
We explain here how two objects in the category of persistence modules might be considered nearby one another. For concreteness, let us consider two interval modules $V^\bullet = \mathbb{Z}_2[a_1,b_1)$ and $U^\bullet = \mathbb{Z}_2[a_2,b_2)$. Recall that an interval module is a special kind of persistence module whose target (here in the category of vector spaces over $\mathbb{Z}_2$) is either the zero element or is of rank one. Specifically, $V^\bullet = \mathbb{Z}_2[a_1,b_1)$ means to construct an $\mathbb{R}$-graded family of vector spaces, $\{V^t\}_{t\in \mathbb{R}}$, so that
\[
V^t = \left\lbrace\begin{matrix}
    \mathbb{Z}_2, & t\in [a_1,b_1)\\
    0, & \text{otherwise}
\end{matrix}\right.,
\]

With this example in mind, we define a map $U^\bullet \to V^\bullet$ that shifts the $\mathbb{R}$-grading by a distance $\delta$. 
A persistence module homomorphism of (persistence) degree $\delta$ is a collection of maps $\varphi^t \colon U^t \to V^{t+\delta}$ such that the following diagram is commutative;
\[\begin{tikzcd} 
 U^s \arrow[r, "u_s^t"] \arrow[d, "\varphi^s"] & U^{t} \arrow[d, "\varphi^t"] \\ V^{s+\delta} \arrow[r, "v_{s+\delta}^{t+\delta}"]& V^{t+\delta}.
\end{tikzcd}\]
A special case is the $\delta$-shift map from a persistence module to itself, i.e. for each real number $t$, $\mathbb{I}_V^\delta \colon V^t \to V^{t+\delta}$ is a persistence module homomorphism of degree $\delta$. We say that two persistence modules, $U^\bullet$ and $V^\bullet$ are $2\delta$-interleaved if there exist homomorphisms $\varphi^t \colon U^t\to V^{t+\delta}$ and $\psi^t \colon V^t \to U^{t+\delta}$, both of degree $\delta$, such that $\psi^{t+\delta} \circ \varphi^t = \mathbb{I}_U^{2\delta}$ and $\varphi^{t+2\delta} \circ \psi^t = \mathbb{I}_V^{2\delta}$.

This allows us to equip the functor category of persistence modules with an extended pseudometric.

\begin{definition} The \emph{interleaving distance} between two persistence modules is given by
    \[d(V^\bullet, U^\bullet) = \inf\{\epsilon \,\vert\, V^\bullet \text{ and } U^\bullet \text{ are } 2\epsilon\text{-interleaved}\},\]
    and we set $d(V^\bullet, U^\bullet) =\infty$ if $V^\bullet$ and $U^\bullet$ are not $\epsilon$-interleaved for any $\epsilon\geq0$.
\end{definition}

See \cite{bubenik2014categorification} for an explanation of why this function constitutes a pseudometric on persistence modules.

\subsubsection{Barcodes} The most important example of a persistence module is an interval module, because the Legendrian contact homology of a knot can be uniquely expressed as a direct sum of interval modules. This claim follows from Theorem~\ref{thm:decomposition}, above, and the following.

\begin{theorem}\label{thm:uniqueness}
\emph{(Uniqueness theorem of Krull, Remak, Schmidt, and Azumaya cf.  \cite{chazal2016strcture})} Suppose that a persistence module $V^\bullet$ can be decomposed into interval modules in two ways as
\[ V^\bullet = \bigoplus_{\ell \in L} \mathbb{Z}_2J_\ell = \bigoplus_{m \in M}\mathbb{Z}_2K_m\]
where $L$ and $M$ are index sets that will be finite in the examples at hand, and where each $J_\ell$ and each $K_m$ is an interval of real numbers. Then there is a bijection $\sigma \colon L \to M$ satisfying $J_\ell = K_{\sigma(\ell)}$ for each $\ell \in L$.
\end{theorem}

We call the decomposition of the LCH persistence module into interval modules a \emph{barcode}.  The uniqueness theorem will permit us to speak of ``the'' barcode of a Legendrian knot. In practice, barcodes for Legendrian knots which admit augmentations may be drawn as a collection of some finite-length intervals and some infinite-length intervals. This is an information-dense way of drawing the persistent Legendrian contact homology. Examples of barcodes as drawings are given in Figure~\ref{fig:toyunknot(b)} and Figure~\ref{fig:toytrefoil(b)}, and the whole of Section~\ref{sec:toys} is devoted to explaining how to compute barcodes, both in general and in specific examples.

 Theorem~\ref{thm:uniqueness} does say that we can permute the intervals in the interval module decomposition that we are calling the barcode. The uniqueness theorem \emph{does not} say that the barcode is insensitive to the assignment of heights, as described in the introduction. 

 In this work, when we apply the interleaving distance to two barcodes, we mean that we are applying the interleaving distance to the persistence module associated to each.

\subsection{Legendrian contact homology}\label{background:lch}
We now establish notation for the Legendrian contact homology of a Legendrian knot $\Lambda\subset\mathbb{R}^3$, constructed from a Lagrangian diagram planar isotopic to  $\Pi(\Lambda)\subset\mathbb{R}^2$.  This construction takes the form of a differential graded algebra --- known as the Chekanov-Eliashberg DGA --- the homology of which is an invariant of the Legendrian isotopy class of $\Lambda$.

Legendrian contact homology fits into the symplectic field theory (SFT) paradigm due to Eliashberg-Givental-Hofer \cite{eliashberg2000introduction}, with the Chekanov-Eliashberg DGA carefully defined in \cite{chekanov2002differential}.  Generalizations later appeared in \cite{etnyre2002invariants}, \cite{ekholm2005orientations}, \cite{ekholm2013knot}, \cite{ng2013satellites}, \cite{ekholm2017duality}, but we will follow \cite{chekanov2002differential} in presenting the DGA over $\mathbb{Z}_2$.

Viewed as a $2$-plane distribution in the tangent bundle of $\mathbb{R}^3$, we denote the standard contact structure as $\xi_{std} = \text{ker}(\alpha_{std})$. The manifold $\mathbb{R}^3$ together with the standard contact structure is a contact manifold, written  $(\mathbb{R}^3,\xi_{std})$.

\subsubsection{Legendrian knots and their diagrams}\label{sec:legs}

As discussed in Section~\ref{sec:intro}, the two most important projections associated to a Legendrian knot $\Lambda\subset\mathbb{R}^3$ are the Lagrangian projection $\Pi(\Lambda)$ and the front projection $F(\Lambda)$.  We will be especially interested in the Lagrangian projection.

While $\Pi(\Lambda)\subset\mathbb{R}^2$ is an immersed curve, we point out that not every closed, immersed curve in $\mathbb{R}^2$ can be realized as a Lagrangian projection.  For instance, a closed, immersed curve which encloses a nonzero signed area is not the Lagrangian projection of any Legendrian knot.  Nonetheless, it is standard practice to represent a Legendrian $\Lambda$ by a knot diagram which is planar isotopic to $\Pi(\Lambda)$, but which may not lift to a Legendrian in $\mathbb{R}^3$.  The distinction between diagrams which lift to Legendrians and those which do not will be important to us, so we establish the following vocabulary.

\begin{definition}
A \emph{Lagrangian projection diagram} (or \emph{Lagrangian diagram}) is a knot diagram in $\mathbb{R}^2$ which is planar isotopic to a Lagrangian projection $\Pi(\Lambda)$ of some Legendrian knot $\Lambda\subset\mathbb{R}^3$.  We assume that $\Pi(\Lambda)$ retains the crossing information of $\Lambda$.
\end{definition}

Where smooth knot diagrams have a strong Reidemeister theorem --- namely, any Reidemeister move on a smooth knot diagram may be realized by a smooth isotopy of knots --- Lagrangian projection diagrams have only a weak Reidemeister theorem.

\begin{theorem}\label{thm:LegReid}\emph{({c.f. \cite[Section 6]{chekanov2002differential}})}
If $\Lambda,\Lambda'$ are Legendrian isotopic Legendrian knots in $\mathbb{R}^3$, then their Lagrangian projection diagrams $\Pi(\Lambda),\Pi(\Lambda')$ are related by a sequence of Legendrian Reidemeister moves, depicted in Figure~\ref{fig:reidemeister-moves}.
\end{theorem}

\begin{figure}
    \centering
    \begin{subfigure}[b]{0.3\textwidth}\vspace{1em}
        \centering
        \includegraphics[scale=0.4]{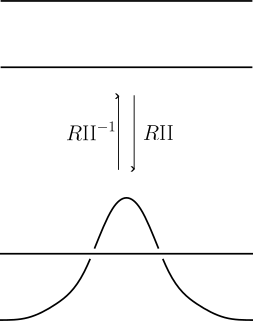}
        \caption{}
        \label{fig:RII}
    \end{subfigure}
    \begin{subfigure}[b]{0.3\textwidth}\vspace{1em}
        \centering
        \includegraphics[scale=0.4]{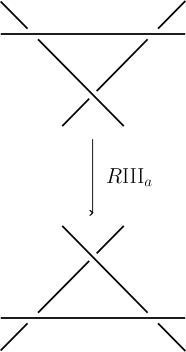}
        \caption{}
        \label{fig:RIIIa}
    \end{subfigure}
    \begin{subfigure}[b]{0.3\textwidth}\vspace{1em}
        \centering
        \includegraphics[scale=0.4]{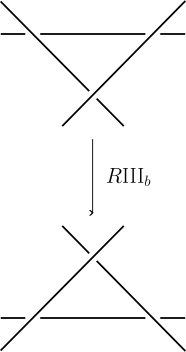}
        \caption{}
        \label{fig:RIIIb}
    \end{subfigure}
    \caption{The three Legendrian Reidemeister moves in the Lagrangian projection. Note that we call the move in which a bigon is removed the ``inverse'' move II. This terminology will be important in Section~\ref{sec:proof}. }
    \label{fig:reidemeister-moves}
\end{figure}

We will see that Legendrian contact homology may be defined from a Lagrangian projection diagram, while persistent Legendrian contact homology \emph{a priori} requires an honest Lagrangian projection.

\subsubsection{The Chekanov-Eliashberg DGA}\label{sec:CEDGA}
The definition of the Chekanov-Eliashberg DGA accepts as input an oriented Legendrian knot $\Lambda\subset(\mathbb{R}^3,\xi_{\std})$.  Important to our work in this paper is that a small Legendrian isotopy may be applied to $\Lambda$ to ensure that its Lagrangian projection $\Pi(\Lambda)\subset\mathbb{R}^2$ has singularities of only one type: transverse double points, where the strands of $\Pi(\Lambda)$ meet orthogonally.  In fact, the starting point of our definition may be any Lagrangian projection diagram for $\Lambda$ with this property.

The algebra $\mathcal{A}_\Lambda$ underlying the Chekanov-Eliashberg DGA associated to $\Lambda$ is generated over $\mathbb{Z}_2$ by $q_1,\ldots,q_n$, where each generator $q_i$, $1\leq i\leq n$, corresponds to a double point of $\Pi(\Lambda)$.  The double points of $\Pi(\Lambda)$ correspond to \emph{Reeb chords} in $\mathbb{R}^3$ with endpoints on $\Lambda$ --- these are arcs parallel to the $z$-axis, and serve as critical points of the Reeb action functional described in Section~\ref{sec:intro}.

The algebra $\mathcal{A}_\Lambda$ is associative, but not commutative.  We refer the reader to \cite{etnyre2020legendrian} for details about the grading and differential on $\mathcal{A}_\Lambda$, but fix notation for these as follows: the grading of a word $w\in\mathcal{A}_\Lambda$ is denoted $|w|\in\mathbb{Z}$, while its differential is denoted $\partial_\Lambda(w)\in\mathcal{A}_\Lambda$.  Up to the small Legendrian isotopy which may be required at the beginning of the construction, $(\mathcal{A}_\Lambda,\partial_\Lambda)$ is a well-defined differential graded algebra.

As a DGA, $(\mathcal{A}_\Lambda,\partial_\Lambda)$ is not an invariant of the Legendrian isotopy type of $\Lambda$.  For instance, a Legendrian isotopy which corresponds to a double point move in $\Pi(\Lambda)$ will either introduce or delete two generators of the DGA.  If, however, $\Lambda,\Lambda'$ are Legendrian isotopic, then the DGAs $(\mathcal{A}_{\Lambda},\partial_{\Lambda})$ and $(\mathcal{A}_{\Lambda'},\partial_{\Lambda'})$ are \emph{stable tame isomorphic}.  Because our proof of Theorem~\ref{thm:main-invariance-result} requires some explicit manipulations of these DGAs, we recall now what it means to be stable tame isomorphic.  The basic story of Legendrian contact homology may be digested without these technical definitions, however.

First, stabilization accounts for the possibility of introducing superfluous generators.

\begin{definition}
For any positive integer $k$, the \emph{grading $k$ stabilization} of a Chekanov-Eliashberg DGA $(\mathcal{A}_\Lambda=\mathbb{Z}_2\langle q_1,\ldots,q_n\rangle,\partial_\Lambda)$ is the DGA whose
\begin{enumerate}
    \item algebra $S_k\mathcal{A}_\Lambda$ is generated over $\mathbb{Z}_2$ by $e_k,e_{k-1},q_1,\ldots,q_n$;
    \item grading is given by $|e_k|=k$, $|e_{k-1}|=k-1$, and agrees with the grading of $\mathcal{A}_\Lambda$ for other generators;
    \item differential $\partial_\Lambda$ satisfies $\partial_\Lambda(e_k)=e_{k-1}$ and agrees with $\partial_\Lambda$ on other generators. 
\end{enumerate}
\end{definition}

After sufficiently many stabilizations, the Chekanov-Eliashberg DGAs of a pair of Legendrian isotopic knots become tame isomorphic.

\begin{definition}
An \emph{elementary automorphism} of a Chekanov-Eliashberg DGA $(\mathcal{A}_\Lambda,\partial_\Lambda)$ with underlying algebra $\mathbb{Z}_2\langle q_1,\ldots,q_n\rangle$ is a chain map $\phi\colon(\mathcal{A}_\Lambda,\partial_\Lambda)\to(\mathcal{A}_\Lambda,\partial_\Lambda)$ of the form
\[
\phi(q_j) = \pm q_j + u,
\qquad
\phi(q_i) = q_i, i\neq j,
\]
for some $1\leq j\leq n$ and some $u\in\mathbb{Z}_2\langle q_1,\ldots,\widehat{q_j},\ldots,q_n\rangle$.  A \emph{tame isomorphism} between two Chekanov-Eliashberg DGAs $(\mathcal{A}_\Lambda=\mathbb{Z}_2\langle q_1,\ldots,q_n\rangle,\partial_\Lambda)$ and $(\mathcal{A}'_\Lambda=\mathbb{Z}_2\langle q'_1,\ldots,q'_n\rangle,\partial'_\Lambda)$ is the result of composing some number of elementary automorphisms of $(\mathcal{A}_\Lambda,\partial_\Lambda)$ with the isomorphism $q_i\mapsto q'_i$.
\end{definition}

We say that a pair of Chekanov-Eliashberg DGAs is stable tame isomorphic if they become tame ismorphic after sufficiently many stabilizations have been applied to each of them.  The most important remark at present is that stable tame isomorphisms are chain homotopy equivalences, and thus the \emph{Legendrian contact homology}
\[
\LCH_*(\Lambda):= H_*(\mathcal{A}_\Lambda,\partial_\Lambda)
\]
is a well-defined invariant of the Legendrian isotopy type of $\Lambda$.

\subsubsection{Linearization}\label{linearize}
In general, $\LCH_*(\Lambda)$ can be difficult to compute and need not have finite rank.  This can be resolved by considering the \emph{linearized Legendrian contact homology}, which we now define.

First, an \emph{augmentation} of $(\mathcal{A}_\Lambda,\partial_\Lambda)$ is an algebra homomorphism $\epsilon\colon\mathcal{A}_\Lambda\to\mathbb{Z}_2$ with the following properties:
\begin{itemize}
    \item $\epsilon(a)=0$, for any $a\in\mathcal{A}_\Lambda$ of nonzero grading;
    \item $\epsilon\circ\partial_\Lambda = 0$.
\end{itemize}
There are Legendrian knots $\Lambda$ for which $(\mathcal{A}_\Lambda,\partial_\Lambda)$ does not admit an augmentation --- for instance, those with $\rot(\Lambda)\neq 0$ --- but the property of admitting an augmentation is invariant under stable tame isomorphism.

Given an augmentation $\epsilon$ of $(\mathcal{A}_\Lambda,\partial_\Lambda)$, we now consider the automorphism $\phi^\epsilon\colon\mathcal{A}_\Lambda\to\mathcal{A}_\Lambda$ defined by
\[
\phi^\epsilon(q_i) := q_i + \epsilon(q_i),
\]
for each generator $q_i$, $1\leq i\leq n$.  Conjugating the differential $\partial_\Lambda$ produces a new differential
\[
\partial^\epsilon_\Lambda := \phi^\epsilon\circ\partial_\Lambda\circ(\phi^\epsilon)^{-1},
\]
making $(\mathcal{A}_\Lambda,\partial_\Lambda^\epsilon)$ a DGA.  The primary advantage held by $\partial^\epsilon_\Lambda$ over $\partial_\Lambda$ is that $\partial^\epsilon_\Lambda(q_i)$, which is a formal sum of words in $q_1,\ldots,q_n$, has constant term given by $(\epsilon\circ\partial_\Lambda)(q_i)=0$, for each $q_i$, $1\leq i\leq n$.  This allows us to define a linear map $\partial_1^\epsilon\colon A_\Lambda\to A_\Lambda$ by
\[
\partial_1^\epsilon(q_i) := \pi_1(\partial^\epsilon_\Lambda(q_i)),
\]
where $A_\Lambda$ is the $\mathbb{Z}_2$-vector space with basis $\{q_1,\ldots,q_n\}$ and $\pi_1\colon\mathcal{A}_\Lambda\to A_\Lambda$ is the linear map which sends to 0 any words of length not equal to 1.

It is straightforward to verify that $\partial_1^\epsilon$ is a differential on $A_\Lambda$, and we define the \emph{linearized Legendrian contact homology of $\Lambda$ with respect to $\epsilon$} to be
\[
\LCH_*^\epsilon(\Lambda) := H_*(A_\Lambda,\partial_1^\epsilon).
\]
It is important to observe that $\LCH_*^\epsilon(\Lambda)$ does, in general, depend on the choice of augmentation $\epsilon$.  However, we may still extract a Legendrian isotopy invariant.

\begin{theorem}[{\cite[Theorem 3.3, Lemma 4.1]{chekanov2002differential}}]\label{thm:linearized-invariant}
The collection of linearized Legendrian contact homologies
\[
\{\LCH_*^\epsilon(\Lambda)\,|\,\epsilon\colon\mathcal{A}_\Lambda\to\mathbb{Z}_2\text{~is an augmentation}\}
\]
is a Legendrian isotopy invariant of $\Lambda$.
\end{theorem}

\subsubsection{Filtering the Chekanov-Eliashberg DGA}\label{sec:background:filtering}

Computing a persistent version of (linearized) Legendrian contact homology will require a filtration at the chain level --- that is, on the DGA $(\mathcal{A}_\Lambda,\partial_\Lambda)$ or the differential vector space $(A_\Lambda,\partial_1^\epsilon)$.  A natural candidate for such a filtration is the \emph{action} or \emph{Reeb height} discussed in Section~\ref{sec:intro}.  In particular, the assignment
\begin{equation}\label{eq:height}
h(q) := \int_\gamma\alpha_{\mathrm{std}} = \int_\gamma dz-ydx,
\end{equation}
where $\gamma$ is the Reeb chord with ends on $\Lambda$ which projects to $q$, fits into the Floer-theoretic perspective which first motivated Legendrian contact homology.  For a fixed (generic) Legendrian $\Lambda$, this is a perfectly acceptable filtration, and one may proceed with the computation of persistent homology.

In practice, however, this definition is of limited utility, since most computations of Legendrian contact homology accept as input a Legendrian isotopy class, rather than a fixed Legendrian knot.  For instance, consider that $\LCH_*^\epsilon(\Lambda)$ is typically computed from the Lagrangian projection $\Pi(\Lambda)$.  Given $\Pi(\Lambda)\subset\R^2$, we may recover $\Lambda\subset\mathbb{R}^3$ up to $z$-translation, and thus recover its Reeb heights; but an arbitrary Lagrangian projection diagram of $\Lambda$ will not generally lift to a Legendrian knot in $\mathbb{R}^3$, leaving us with the question of how to define heights from a Lagrangian diagram.

Given a Lagrangian diagram $D\subset\mathbb{R}^2$, we will refer to the bounded components of $\R^2\setminus D$ as \emph{area patches}.  Each area patch $P$ has corners corresponding to Reeb chords, and these corners are assigned \emph{Reeb signs}, defined as follows: as we traverse $\partial P$ with the boundary orientation, a positive Reeb sign is assigned to each corner where we move from an understrand to an overstrand; other corners are assigned negative Reeb signs.

Given an explicit Lagrangian projection $\Pi(\Lambda)$, one may use Stokes' theorem together with the Legendrian condition (\ref{eq:LegCondition}) to find that the (unsigned) area of an area patch is equal to the sum of the Reeb heights of its corners, weighted by the Reeb signs.  This area may change under planar isotopy, but must remain positive.  Thus we may associate a linear inequality to each area patch of a Lagrangian projection diagram.  For instance, the area patches in Figure~\ref{fig:toyunknot(a)} each lead to the inequality $h(q)>0$, while the topmost area patch of Figure~\ref{fig:toytrefoil(a)} has inequality
\[
h(q_1)-h(q_3)-h(q_4)-h(q_5) > 0.
\]
When we speak of the \emph{area inequalities} of a Lagrangian diagram $D$ we are referring to the collection of linear inequalities resulting from its area patches.  We recall that the goals of the present paper include:
\begin{enumerate}
    \item an algorithmic solution to the area inequalities, given a Lagrangian diagram $D$;
    \item a description of how the solution set of the area inequalities changes under Legendrian isotopy.
\end{enumerate}

\section{Toy examples}\label{sec:toys}

This section is concerned with computing persistent Legendrian contact homology in some simple cases. In Section~\ref{barcodegeneral} we describe the general procedure for computing persistent LCH from a Lagrangian projection diagram, and then we present two instructive examples.

\subsection{The general case}\label{barcodegeneral}
We outline the steps needed to construct a barcode from a Lagrangian diagram, provided the Chekanov-Eliashberg DGA admits an augmentation.  Throughout, the corresponding Legendrian isotopy type is denoted by $\Lambda$ and the crossings of the diagram are denoted by $q_1,\ldots,q_n$.

\begin{enumerate}
    \item First, calculate the non-linear Chekanov-Eliashberg DGA $\mathcal{A}_\Lambda$.  Pick a valid augmentation $\epsilon$ and use this to produce a chain complex of $\mathbb{Z}_2$ vector spaces, denoted $(A, \partial_1^\epsilon)$, following the linearization procedure described in Section~\ref{linearize}. 

    \item Assign a height $h(q_i)>0$ to each generator $q_i$ of $A$.  Because this assignment is not included in the data of a Lagrangian diagram, one must make a choice; we present one algorithm for making this choice in Section~\ref{sec:flooding}. \label{step:assign-heights} 

    \item Use the assignment made in Step~\ref{step:assign-heights} to upgrade $A$ to a persistence module of a real parameter.  Namely, for each $\ell \in \mathbb{R}$, let $A^\ell$ denote the subspace of $A$ that is generated by those $q_i$ satisfying $h(q_i)\leq \ell$ and define
    \[
    A^\bullet = \bigoplus_{\ell\in\mathbb{R}}A^\ell.
    \]

    \item Compute the homology of each chain complex $(A^\ell,\partial_1^\epsilon)$ and define the persistence module
    \[
    F^\bullet \LCH^\epsilon_\ast(\Leg) = \bigoplus_{\ell\in\mathbb{R}}H_*(A^\ell, \partial_1^\epsilon).
    \]
    A transfer map on this persistence module will be induced from applying the homology functor to the corresponding transfer map of $A^\bullet$. Note that each vector space $F^t \LCH^\epsilon_\ast(\Leg)$ also has a $\mathbb{Z}$-grading, given by taking homology in different degrees.  We refer to $F^\bullet \LCH^\epsilon_\ast(\Leg)$ as the \emph{persistent Legendrian contact homology} of $\Lambda$ with the auxiliary data of an augmentation $\epsilon$ and height assignment $h$.
\end{enumerate}

It is our decomposition of $F^\bullet \LCH^\epsilon_\ast(\Leg)$ into interval modules which allows us to depict the persistent LCH as a barcode.  Moreover, we will typically decorate our barcodes with (linear combinations of) generators $q_i$, to indicate choices of homology bases.

\subsection{The standard unknot}

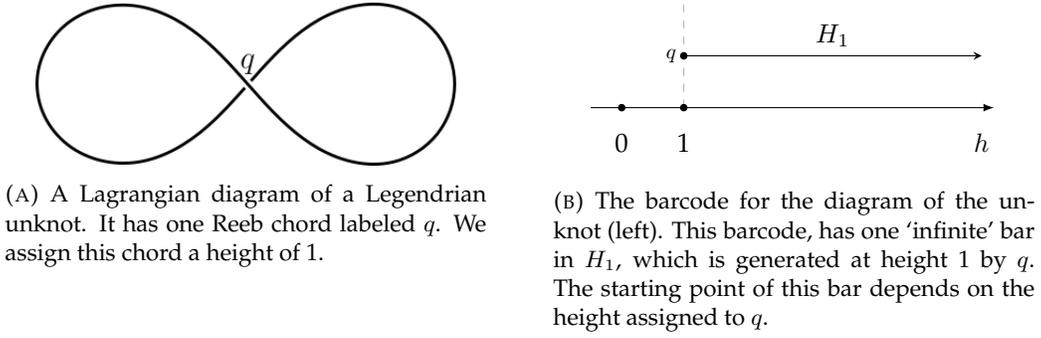
\begin{figure}
    \label{fig:toyunknot}
    \centering
    
    \begin{subfigure}[t]{0.4\textwidth}\vspace{1em}
        \centering
        \includegraphics[scale=0.4]{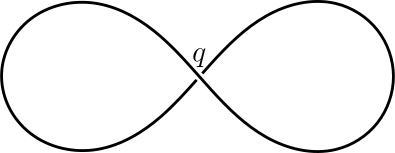}
        \caption{A Lagrangian diagram of a Legendrian unknot. It has one Reeb chord labeled $q$. We assign this chord a height of 1.}
        \label{fig:toyunknot(a)}
    \end{subfigure}\hspace{0.05\textwidth}
    \begin{subfigure}[t]{0.4\textwidth}\vspace{1em}
    \centering
        \begin{tikzpicture}[xscale=1.65,yscale=2.3]
            \draw[-latex] (-.25,0) -- (3,0); \draw[dashed, lightgray] (0.5,0) -- (0.5,.5);
            \draw[dashed, lightgray] (0.5,0) -- (0.5,.6);
            \foreach \x in  {0,0.5}
            \node[blackpoint] at (\x,0){};
            \node[textnode] at (0,-.2) {0};
            \node[textnode] at (0.5,-.2) {1};
            \node[textnode] at (2.9,-.2) {$h$};
            \draw[-stealth] (0.5,.3) node[blackpoint]{} node[left,scale=0.8]{$q$} --  node[above]{$H_1$} (2.9,.3) {};
        \end{tikzpicture}
        \caption{The barcode for the diagram of the unknot (left). This barcode, has one `infinite' bar in $H_1$, which is generated at height 1 by $q$. The starting point of this bar depends on the height assigned to $q$.}
        \label{fig:toyunknot(b)}
    \end{subfigure}
    \caption{The Legendrian unknot and its barcode.}
\end{figure}

For the standard Legendrian unknot whose Lagrangian diagram is pictured in Figure~\ref{fig:toyunknot(a)}, the algebra $\mathcal{A}_\Lambda$ has a single generator $q$, with $|q|=1$, and a simple computation (c.f. \cite[Example 3.5]{etnyre2020legendrian}) shows that $\partial_\Lambda q = 1$.  There is only one augmentation $\epsilon\colon\mathcal{A}_\Lambda\to\mathbb{Z}_2$, and it satisfies $\epsilon(q)=0$.  It follows that the augmented differential $\partial_1^{\epsilon}$ is trivial.  By declaring $h(q)=1$, we may upgrade $A=\mathbb{Z}_2\langle q\rangle$ to a persistence module $A^\bullet$, where
\[
A^\ell = \left\{\begin{matrix}
    0 & \ell < 1\\
    \mathbb{Z}_2\langle q\rangle & 1\leq \ell
\end{matrix}\right.
\]
Because $\partial_1^{\epsilon}$ is trivial, $H_*(A^\ell,\partial_1^\epsilon)\cong A^\ell$, for every $\ell$, and we see that $F^\bullet \LCH^\epsilon_\ast(\Leg)$ is the interval module $\mathbb{Z}_2[1,\infty)$.  The persistent LCH of the unknot is depicted in~Figure~\ref{fig:toyunknot(b)}, with the label $q$ on our infinite bar representing our choice of homology basis $[q]$, as well as the fact that $q$ is ``born" into existence when we begin considering generators of height 1 or greater.

\subsection{The standard trefoil}\label{sec:trefoilex}
Our second (and more interesting) toy example is the standard Legendrian trefoil whose Lagrangian diagram is pictured in Figure~\ref{fig:toytrefoil(a)}.  With the generators as labeled there, one may compute (c.f. \cite[Example 3.6]{etnyre2020legendrian}) the gradings
\[
|q_1| = |q_2| = 1
\quad\text{and}\quad
|q_3| = |q_4| = |q_5| = 0
\]
and the differential $\partial_\Lambda$
\begin{align*}
\partial_\Lambda q_1&= 1 + q_5 + q_5q_4q_3 + q_3 \\
\partial_\Lambda q_2&= 1 + q_3 + q_3q_4q_5 + q_5 \\
\partial_\Lambda q_3 &= \partial_\Lambda q_4 = \partial_\Lambda q_5 = 0.
\end{align*}

\begin{figure}
    \label{fig:toytrefoil}
    \centering
    
    \begin{subfigure}[t]{0.4\textwidth}\vspace{1em}
        \centering
        \includegraphics[scale=0.4]{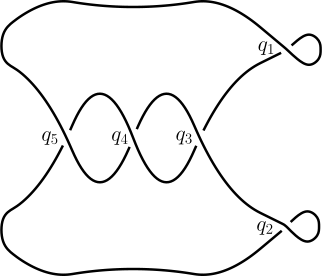}
        \caption{A Lagrangian diagram of a Legendrian trefoil derived from Ng's resolution applied to a front diagram of the trefoil. It has five Reeb chords labeled $q_1$ through $q_5$. $q_1$ and $q_2$ are both assigned height 4. The other three chords are assigned height 1.}
        \label{fig:toytrefoil(a)}
    \end{subfigure}\hspace{0.05\textwidth}
    \begin{subfigure}[t]{0.4\textwidth}\vspace{1em}
        \centering
            \begin{tikzpicture}[xscale=1.65,yscale=2.3]
                    \draw[-latex] (-.25,0) -- (3,0); \draw[dashed, lightgray] (0.5,0) -- (0.5,.5);
                    \draw[dashed, lightgray] (0.5,0) -- (0.5,.8);
                    \draw[dashed, lightgray] (2,0) -- (2,.9);
                    \foreach \x in  {0,0.5,2}
                    \node[blackpoint] at (\x,0){};
                    \node[textnode] at (0,-.2) {0};
                    \node[textnode] at (0.5,-.2) {1};
                    \node[textnode] at (2,-.2) {4};
                    \node[textnode] at (2.9,-.2) {$h$};
                    \draw[-stealth] (0.5,.2) node[blackpoint]{} node[left,scale=0.8]{$q_5$} -- (2.9,.2) {};
                    \draw[-stealth] (0.5,.3) node[blackpoint]{} node[left,scale=0.8]{$q_4$} -- (2.9,.3) {};
                    \draw[-stealth] (2,.6) node[blackpoint]{} node[left,scale=0.8]{$q_1+q_2$} --  node[above]{$H_1$} (2.9,.6) {};
                    \draw (0.5,.4) node[blackpoint]{} node[left,scale=0.8]{$q_3+q_5$} --  (2,.4) node[whitepoint]{} node[right,scale=0.8] {$q_1,q_2$};
                    \draw (0.0,.1) -- ++(-.1,0) -- node[left]{$H_0$}  ++(0,.4) -- ++(.1,0);
            \end{tikzpicture}
        \caption{The barcode for the diagram of the trefoil (left). This barcode, has one infinite bar in $H_1$, which is created at height 4, and two infinite bars in $H_0$, both created at height 1. There is also a finite bar in $H_0$ created at height 1 and destroyed at height 4. The starting point and length of these bars depend on the heights we assign to the Reeb chords.}
        \label{fig:toytrefoil(b)}
    \end{subfigure}    
    \caption{The Legendrian trefoil and its barcode}
\end{figure}
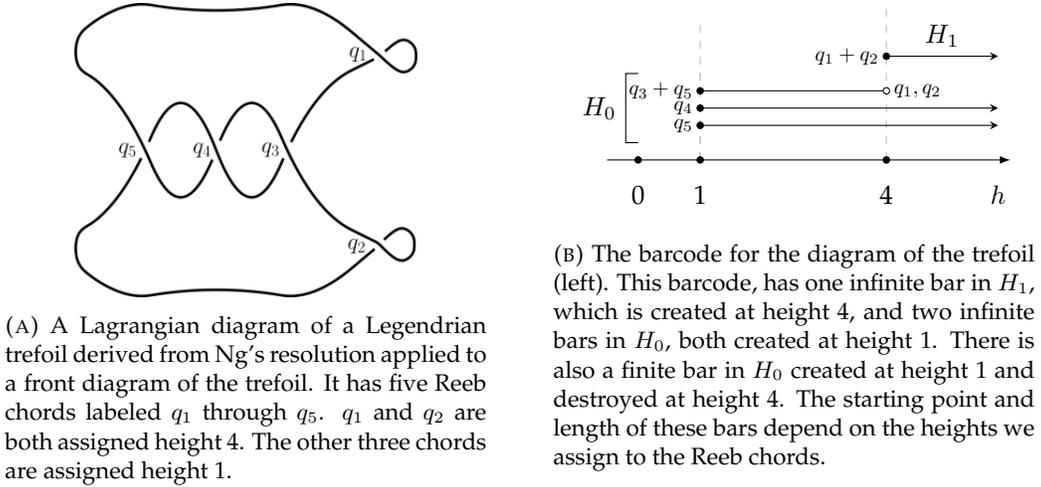

The DGA $(\mathcal{A}_\Lambda,\partial_\Lambda)$ admits five augmentations to $\mathbb{Z}_2$ (c.f. \cite[Example 4.6]{etnyre2020legendrian}), but one may compute that all of these give the linearized differential
\begin{align*}
\partial_{1}^{\epsilon} q_1 &=  q_5  + q_3 \\
\partial_{1}^{\epsilon} q_2 &= q_5 +q_3 \\
\partial_{1}^{\epsilon} q_3 &= \partial_{1}^{\epsilon} q_4 = \partial_{1}^{\epsilon} q_5 = 0.
\end{align*}
Through a process to be described in Section~\ref{sec:flooding}, we produce the heights $h(q_1) = h(q_2) = 4$ and $h(q_3) = h(q_4) = h(q_5) = 1$.  Thus we have
\[
A^\ell = \left\{\begin{matrix}
    0 & \ell < 1\\
    \mathbb{Z}_2\langle q_3,q_4,q_5\rangle & 1\leq \ell < 4\\
    \mathbb{Z}_2\langle q_1,\ldots,q_5\rangle & 4\leq \ell
\end{matrix}\right.
\]
Notice that for $\ell<4$, $\partial_1^\epsilon$ restricts to $A^\ell$ as the zero map, meaning that $H_*(A^\ell,\partial_1^\epsilon)\cong A^\ell$ for $\ell<4$.  For $\ell\geq 4$, the three degree-0 generators of $A^\ell$ are joined by two generators of degree 1, and the sum of these degree-1 generators is trivial in homology.  That is,
\[
H_k(A^\ell,\partial_1^\epsilon) \cong \left\{\begin{matrix}
    \mathbb{Z}_2^2, & k=0\\
    \mathbb{Z}_2, & k=1
\end{matrix}\right.,
\quad\text{for~}\ell\geq 4,
\]
and we may now express the persistent Legendrian contact homology as a direct sum of interval modules:
\[
F^\bullet \LCH^\epsilon_k(\Leg) \cong \left\{\begin{matrix}
    \mathbb{Z}_2[0,\infty)\oplus\mathbb{Z}_2[0,\infty)\oplus\mathbb{Z}_2[0,4), & k=0\\
    \mathbb{Z}_2[4,\infty), & k=1
\end{matrix}\right.
\]
Figure~\ref{fig:toytrefoil(b)} presents the persistent LCH in barcode form.  Once again we have labeled the left endpoints of our bars with an arbitrary choice of homology basis.  The right endpoint of the finite bar is labeled with the generators whose birth causes the change in homology.

\section{Proof of Theorem~\ref{thm:main-invariance-result}}\label{sec:proof}

Because of Theorem~\ref{thm:LegReid}, we know that any Legendrian isotopy of Legendrian knots can be realized in the Lagrangian projection as some sequence of Legendrian Reidemeister moves and planar isotopies. We wish to show that the persistent Legendrian contact homology barcode is an invariant of a Lagrangian projection diagram of a Legendrian knot. It is for this reason that we consider the Lagrangian projection of a knot modulo planar isotopies, and we study how the barcode changes after a single Legendrian Reidemeister move is applied to the diagram. 

In this section we first consider those tame isomorphisms on the Chekanov-Eliashberg DGA which are \emph{semimonotonic} --- to be defined presently --- and argue that semimonotonic isomorphisms will induce a $2\delta$-interleaving on the persistent LCH.  See Proposition~\ref{prop:2delta-interleaving}.  We then argue that the tame isomorphism induced by a single Reidemeister move will be semimonotonic.  This work is done in Sections~\ref{sec:R2} through \ref{sec:R3a}.

Henceforth we will use the word \emph{bigon} to mean a bounded (interior) region of a Lagrangian knot diagram that meets exactly two crossings. An example is drawn in Figure~\ref{fig:thm1.1RII^{-1}move(a)}, below. The naming conventions for the Legendrian Reidemeister moves are shown in Figure~\ref{fig:reidemeister-moves}.

\subsection{The Homological algebra of Reidemeister isotopies}\label{sec:HomAlg}

Throughout this section we consider a filtered Chekanov-Eliashberg DGA $\mathcal{A}_\Lambda$ generated over $\mathbb{Z}_2$ by $q_1,\ldots,q_n$.  We will use the word \emph{height} to refer to the filtration of $\mathcal{A}_\Lambda$, and remind the reader that a Legendrian isotopy of the underlying knot $\Lambda$ will induce a \emph{stable tame isomorphism}.  See Subsection~\ref{sec:CEDGA}.

\begin{definition}
An elementary automorphism $\Bar{\phi}:q_i\mapsto q_i+u$ is called \emph{semimonotonic} if $u$ is a word in generators of height less than $q_i$.  A tame isomorphism is semimonotonic if the elementary automorphisms which witness its tameness may be taken to be semimonotonic.
\end{definition}

Recall from Subsection~\ref{linearize} that the computation of linearized LCH requires conjugating the differential $\partial_\Lambda$ by a certain automorphism of $\mathcal{A}_\Lambda$.  Any tame isomorphism $\Bar{\phi}:\Alg_\Leg\rightarrow \Alg_\Leg$ analogously induces a map
\[
\Bar{\phi}_1^{\epsilon}:(A_\Leg,\partial_1^\epsilon)\rightarrow (A_\Leg,\Bar{\phi}_1^{\epsilon}\circ \partial_1^{\epsilon}\circ (\Bar{\phi}_1^{\epsilon})^{-1}).
\]
on the linearized DGA $(A_\Lambda,\partial_1^\epsilon)$ with respect to an augmentation $\epsilon\colon\mathcal{A}_\Lambda\to\mathbb{Z}_2$.  The following lemma verifies that if $\Bar{\phi}$ is semimonotonic, then we in fact have an induced map on persistence chain complexes.

\begin{lemma}
    \label{lem:semimonotonicity}
    Let $\Bar{\phi}:\Alg_\Leg\rightarrow \Alg_\Leg$ be a tame isomorphism. If $\Bar{\phi}$ is semimonotonic, then $(A_\Leg^\bullet,\Bar{\phi}_1^{\epsilon}\circ \partial_1^\epsilon\circ (\Bar{\phi}_1^{\epsilon})^{-1})$ is a well-defined persistence module of chain complexes and there is an induced isomorphism of persistence modules
        \[\phi_1^{\epsilon,\bullet}:(A_\Leg^\bullet,\partial_1^{\epsilon})\rightarrow (A_\Leg^\bullet,\Bar{\phi}_1^{\epsilon}\circ \partial_1^\epsilon\circ (\Bar{\phi}_1^{\epsilon})^{-1}).\]
\end{lemma}

\begin{proof}
    First, consider the case when $\phi$ is a semimonotonic elementary automorphism with $\phi(q)=q+q_{i_1}\cdots q_{i_k}$. Then
        \[\Bar{\phi}_1^\epsilon(q)=q+\sum_{\ell=1}^k \epsilon(q_{i_1})\cdots \widehat{\epsilon(q_{i_\ell})}\cdots \epsilon(q_{i_k})q_{i_\ell}\]
    and, by definition,
    \[
    h\left(q+\sum_{\ell=1}^k \epsilon(q_{i_1})\cdots \widehat{\epsilon(q_{i_\ell})}\cdots \epsilon(q_{i_k})q_{i_\ell}\right) = \max_\ell(h(q),h(\epsilon(q_{i_1})\cdots \widehat{\epsilon(q_{i_\ell})}\cdots \epsilon(q_{i_k})q_{i_\ell}))=h(q),
    \]
    so that $\Bar{\phi}_1^\epsilon$ is height-preserving. Thus $\Bar{\phi}_1^{\epsilon}\circ \partial_1^\epsilon\circ (\Bar{\phi}_1^{\epsilon})^{-1}$ is height-decreasing so $(A_\Leg^\bullet,\Bar{\phi}_1^{\epsilon}\circ \partial_1^\epsilon\circ (\Bar{\phi}_1^{\epsilon})^{-1})$ is a well-defined chain complex. Clearly, $\Bar{\phi}_1^\epsilon$ induces a persistence isomorphism since $(\Bar{\phi}_1^\epsilon)^{-1}=\Bar{\phi}_1^\epsilon$ over $\Z_2$.

    If $\Bar{\phi}$ is an arbitrary semimonotonic tame isomorphism, then $\Bar{\phi}_1^\epsilon$ is a composition of height-preserving elementary automorphisms by the previous paragraph. For the exact same reasons as in the elementary automorphism case, we see that $(A_\Leg^\bullet,\Bar{\phi}_1^{\epsilon}\circ \partial_1^\epsilon\circ (\Bar{\phi}_1^{\epsilon})^{-1})$ is well-defined and $\Bar{\phi}_1^\epsilon$ induces the desired persistence isomorphism.
\end{proof}

We now see that a composition of semimonotonic elementary automorphisms induces an isomorphism of persistent chain complexes.  But an isotopy of Legendrian knots will induce a tame isomorphism --- obtained by composing some number of elementary automorphisms with the isomorphism $q\mapsto q'$ --- and thus we need a criterion to ensure that the isomorphism $q\mapsto q'$ induces a $2\delta$-interleaving of persistent chain complexes.  The following lemma provides such a criterion.

\begin{lemma}\label{lem:sigma-2delta-interleaving}
Let $(\mathcal{A}_\Lambda,\partial_\Lambda)$ and $(\mathcal{A}'_\Lambda,\partial'_\Lambda)$ be filtered Chekanov-Eliashberg DGAs, with the former generated by $q_1,\ldots,q_n$ and the latter generated by $q_1',\ldots,q_n'$.  Let $\sigma\colon(\mathcal{A}_\Lambda,\partial_\Lambda) \to (\mathcal{A}'_\Lambda,\partial'_\Lambda)$ be the DGA isomorphism defined by $\sigma(q_i)=q_i'$.  If 
\begin{align*}
    \delta\geq\min\{|h'(\sigma(q))-h(q)|\,\text{\emph{ such that }}\,q\in A_\Lambda\},
\end{align*}
where $h$ and $h'$ denote the height filtrations in $\mathcal{A}_\Lambda$ and $\mathcal{A}'_\Lambda$, respectively, then $\sigma$ induces a $2\delta$-interleaving of the persistent linearized contact chain complexes $(A_\Lambda^\bullet,\partial_1^\epsilon)$ and $((A'_\Lambda)^\bullet,(\partial')_1^\epsilon)$.
\end{lemma}
\begin{proof}
For any $t\in\mathbb{R}$ we define
\[
\varphi^t\colon (A_\Lambda^t,\partial_1^\epsilon) \to ((A'_\Lambda)^{t+\delta},(\partial')_1^\epsilon)
\]
by $\varphi^t(q):=v_{h'(\sigma(q))}^{t+\delta}(\sigma(q))$, for any $q\in A_\Lambda^t$.  Here $v_r^s\colon((A'_\Lambda)^r,(\partial')_1^\epsilon)\to((A'_\Lambda)^s,(\partial')_1^\epsilon)$ is the linear transfer map, for any $r\leq s$.  Our assumption on the size of $\delta$ ensures that $h'(\sigma(q))\leq h(q)+\delta\leq t+\delta$, and thus $\varphi^t(q)$ is well-defined.  We similarly define
\[
\psi^t\colon ((A'_\Lambda)^{t},(\partial')_1^\epsilon) \to (A_\Lambda^{t+\delta},\partial_1^\epsilon)
\]
by the formula $\psi^t(q'):=u_{h(\sigma^{-1}(q'))}^{t+\delta}(\sigma^{-1}(q'))$, with $u_r^s$ denoting a transfer map on $A_\Lambda^\bullet$.  Now for any $q\in A_\Lambda^t$ we see that
\[
(\psi^{t+\delta}\circ\varphi^t)(q) = \psi^{t+\delta}(v_{h'(\sigma(q))}^{t+\delta}(\sigma(q))) = u_{h(q)}^{t+2\delta}(q) = \mathbb{I}_{A_\Lambda}^{2\delta}(q).
\]
We similarly see that $\varphi^{t+\delta}\circ\psi^t=\mathbb{I}_{A'_\Lambda}^{2\delta}$. We have exhibited a $2\delta$-interleaving between the persistence modules $A_\Lambda^\bullet$ and $(A'_\Lambda)^\bullet$, as desired.
\end{proof}

The previous two lemmas provide criteria to ensure that a tame isomorphism induces a $2\delta$-interleaving of chain complexes.  Because homology is functorial, these criteria are sufficient to ensure a $2\delta$-interleaving of the resulting homologies $F^\bullet\LCH^\epsilon_\ast(\Leg)$ and $F^\bullet\LCH^\epsilon_\ast(\Leg')$.  We summarize this fact as follows.

\begin{proposition}\label{prop:2delta-interleaving}
Let $\Bar{\phi}\colon(\mathcal{A}_\Lambda,\partial_\Lambda)\to(\mathcal{A}_\Lambda,\partial_\Lambda)$ be a composition of semimonotonic elementary automorphisms, and let $\sigma\colon(\mathcal{A}_\Lambda,\partial_\Lambda) \to (\mathcal{A}'_\Lambda,\partial'_\Lambda)$ be the DGA isomorphism identified in Lemma~\ref{lem:sigma-2delta-interleaving}, satisfying the hypotheses given there.  Then $\phi=\sigma\circ\Bar{\phi}$ induces a $2\delta$-interleaving
\[
(\phi_1^{\epsilon,\bullet})_*\colon F^\bullet\LCH_\ast^\epsilon(\Leg)\to F^{\bullet+\delta}\LCH_\ast^{\epsilon'}(\Leg').
\]
\end{proposition}

Finally, we observe that a filtered Chekanov-Eliashberg DGA may be stabilized in a manner which preserves its homology, up to $2\delta$-interleaving.

\begin{proposition}\label{prop:stabilization-interleaving}
Let $(\mathcal{A}_\Lambda,\partial_\Lambda)$ be a filtered Chekanov-Eliashberg DGA, and let $(S_k\mathcal{A}_\Lambda,\partial_\Lambda)$ be its grading-$k$ stabilization, with additional generators denoted $e_k$ and $e_{k-1}$.  If the filtration of $\mathcal{A}_\Lambda$ is extended to $S_k\mathcal{A}_\Lambda$ in such a way that $0<h(e_k)-h(e_{k-1})<2\delta$, then the persistent linearized contact homologies of $(\mathcal{A}_\Lambda,\partial_\Lambda)$ and  $(S_k\mathcal{A}_\Lambda,\partial_\Lambda)$ are $2\delta$-interleaved.
\end{proposition}
\begin{proof}
For any $t\in\mathbb{R}$ we define homomorphisms
\[
\iota^t\colon (A_\Lambda^t,\partial_1^\epsilon) \to (S_kA_\Lambda^{t+\delta},\partial_1^\epsilon)
\]
and
\[
r^t \colon (S_kA_\Lambda^t,\partial_1^\epsilon) \to (A_\Lambda^{t+\delta},\partial_1^\epsilon).
\]
The former is the natural inclusion operation followed by the transfer map $u_t^{t+\delta}$, while the latter is defined by
\[
r^t(q) := \left\{\begin{matrix}
    u_t^{t+\delta}(q), & q\neq e_k,e_{k-1}\\
    u_t^{t+\delta}(0), & q=e_k,e_{k-1}
\end{matrix}\right..
\]
Notice that $r^{t+\delta}\circ\iota^t=\mathbb{I}_{A_\Lambda}^{2\delta}$, so the induced map on homology is $\mathbb{I}_{H_*(A_\Lambda)}^{2\delta}$.  On the other hand, $\iota^{t+\delta}\circ r^t$ does not agree with $\mathbb{I}_{S_kA_\Lambda}^{2\delta}$, since, for instance, $(\iota^{t+\delta}\circ r^t)(e_{k-1})=0\neq u_t^{t+2\delta}(e_{k-1})$.  However, $e_{k-1}=\partial e_k$ vanishes at the level of homology, and thus $\iota^{t+\delta}_*\circ r^t_*=\mathbb{I}_{H_*(S_kA_\Lambda)}^{2\delta}$.  We conclude that $\iota^t_*$ and $r^t_*$ give a $2\delta$-interleaving of $H_*(A_\Lambda)$ and $H_*(S_kA_\Lambda)$.
\end{proof}

Propositions~\ref{prop:2delta-interleaving} and~\ref{prop:stabilization-interleaving} give our strategy for proving Theorem~\ref{thm:main-invariance-result}.  Namely, any Legendrian isotopy will induce a stable tame isomorphism between Chekanov-Eliashberg DGAs --- that is, a tame isomorphism between stabilizations of the Chekanov-Eliashberg DGAs.  We will prove Theorem~\ref{thm:main-invariance-result} by demonstrating that there exist height assignments before and after the isotopy such that:

\begin{enumerate}
    \item any generator which survives the isotopy sees its height change by at most $\delta$;
    \item any two cancelling generators which are created or destroyed by the isotopy have heights that are at most $2\delta$ apart;
    \item the induced tame isomorphism of DGAs is semimonotonic.
\end{enumerate}
The first two conditions ensure that the hypotheses of Lemma~\ref{lem:sigma-2delta-interleaving} and Proposition~\ref{prop:stabilization-interleaving} are satisfied, while the final condition allows us to apply Lemma~\ref{lem:semimonotonicity}.  In the ensuing four subsections, we apply this strategy to the Legendrian isotopies depicted in Figure~\ref{fig:reidemeister-moves}.

\subsection{The case of a Reidemeister II move}\label{sec:R2}

Consider an arbitrary Lagrangian diagram of a Legendrian knot and consider any possible Reidemeister II move on that diagram. Denote the Lagrangian diagrams before and after the move as $\Pi(\Lambda)$ and $\Pi(\Lambda')$, with $\Pi(\Lambda')$ having two crossings which do not appear in $\Pi(\Lambda)$.  Following Figure~\ref{fig:thm1.1RIImove(c)} and Proposition~\ref{prop:2delta-interleaving}, we refer to these crossings as $a$ and $b$.  From Remark 2.8 and Theorem 4.1 of \cite{kalman2005contact}, we know that the heights of $a$ and $b$ can be made arbitrarily close to one another relative to the heights of the surrounding crossings.  Here the `surrounding crossings' are the vertices of the area patches which meet $a$ or $b$ and \cite[Theorem 4.1]{kalman2005contact} allows us to ensure that $0<h(a)-h(b)<\delta$, for some $\delta>0$ much smaller than the height of any surrounding crossing.

We first examine how the area inequalities associated to $\Pi(\Lambda')$ compare to those of $\Pi(\Lambda)$.  Inequalities associated to area patches not depicted in Figure~\ref{fig:thm1.1RIImove} will persist as we move from $\Pi(\Lambda)$ to $\Pi(\Lambda')$.  The inequalities associated to the top- and bottommost area patches of Figure~\ref{fig:thm1.1RIIIbmove(c)} will differ from those of the top- and bottommost area patches of Figure~\ref{fig:thm1.1RIIIbmove(a)} by the addition or subtraction of $h(a)-h(b)$.  Similarly, the two middle area patches of $\Pi(\Lambda')$ will give area inequalities whose sum differs from that of the middle area patch in $\Pi(\Lambda)$ by $h(a)-h(b)$.

Thus, examining the area patches in $\Pi(\Lambda)$ and $\Pi(\Lambda')$, we see that any inequality given by an area patch in $\Pi(\Lambda)$ is a sum of inequalities from area patches in $\Pi(\Lambda')$, plus possibly the addition or the subtraction of $h(a)-h(b)$, which we assumed to be arbitrarily small.

Thus any valid height assignment for $\Pi(\Lambda')$ with $h(a) < h(b) + \delta$ induces a height assignment for $\Pi(\Lambda)$, up to shifting the heights of surrounding crossings by $\delta$ to compensate.  Because we have assumed $\delta>0$ to be significantly smaller than each of these surrounding heights, this shift presents no problem.

We now describe how the barcode changes through the entire move. First there is a stable tame isomorphism on the level of DGAs. Then there is a reassignment of heights. From \cite{chekanov2002differential} we know that the stable tame isomorphism consists of a stabilization followed by a composition of elementary automorphisms, each of which has the form $q\mapsto q + \sum w_i$, where each $w_i$ is a word in generators of height strictly smaller than $h(q)$. We are identifying this map as a semimonotonic automorphism, so that Lemma~\ref{lem:semimonotonicity} applies.  Our above discussion of height assignments ensures that the height reassignment will induce a shift of the heights of surrounding crossings by at most $h(a) - h(b) < \delta$.  Thus by Lemma~\ref{lem:sigma-2delta-interleaving} and Proposition~\ref{prop:2delta-interleaving}, this induces a $2\delta$-interleaving of the barcodes.

\subsection{The case of an inverse Reidemeister II move}

Recall that we wish to apply Proposition~\ref{prop:2delta-interleaving} to the case at hand. We consider this move separately from the standard Reidemeister II move since it requires a separate argument regarding the existence of height assignments before and after the move which allow us to apply Lemma~\ref{lem:sigma-2delta-interleaving}. However, the argument regarding the induced stable tame isomorphism follows directly from Section~\ref{sec:R2}. For any inverse Reidemeister II move, the standard Reidemeister move II which undoes it gives a stable tame isomorphism $\phi$ that induces a $2\delta$-interleaving. Thus the stable tame isomorphism for the inverse Reidemeister move II $\phi^{-1}$ also induces a $2\delta$-interleaving. The inverse of a Legendrian Reidemeister II move will remove a bigon from the knot diagram, as depicted in Figure~\ref{fig:thm1.1RII^{-1}move(b)}. 

We cannot simply run the preceding argument in reverse; so consider an arbitrary Lagrangian diagram of a Legendrian knot where the inverse of a Legendrian Reidemeister II can be performed. Label the crossings of the (aforementioned, pictured) bigon $a$ and $b$. The argument from Section~\ref{sec:R2} that the heights $h(a)$ and $h(b)$ can be made arbitrarily close to each other still applies, so we can assume that $h(b) - h(a) < \delta$ for some small $\delta$. Since the inverse Reidemeister move is possible, by \cite[Theorem 4.1]{kalman2005contact} we know that the area patch above the bigon $R$ containing the top arc $ab$ is positive, and its area can be assumed to be greater than $\delta$. See Figure~\ref{fig:thm1.1RII^{-1}move(a)} for an illustration of the bigon $R$. This area, and the area patch below $R$ containing the bottom arc $ab$, will only change by $\pm |h(a) - h(b)| \leq \delta$. Consequently any area inequality involving these generators will still be satisfied.

Additionally, that fact that a move II inverse can be done means there exists a height assignment satisfying the positivity of the {area patches $R_a$ and $R_b$}, while also keeping $a$ and $b$ within $\delta$ of each other. 

Starting with such a height assignment for Figure~\ref{fig:thm1.1RII^{-1}move(a)} we note that deleting $a$ and $b$ and keeping all other heights the same gives a valid height assignment for Figure~\ref{fig:thm1.1RII^{-1}move(c)}, since the area upper and lower area regions change by at most $\pm |h(a) - h(b)| \leq \delta$, and although we subtract $a$ from the inequality on the left, we add the remaining generators from the right area patch, which must add up to at least more than $h(a) - \delta$, meaning the area inequality is still satisfied.

Thus there exists a height assignment before and after the move such that the height of all generators apart from $a$ and $b$ remain unchanged. Hence, Lemma~\ref{lem:sigma-2delta-interleaving} and Proposition~\ref{prop:2delta-interleaving} imply a $2\delta$-interleaving of the barcodes that one computes before and after the Reidemeister move.

\begin{figure}
    \centering
\begin{subfigure}[t]{\textwidth}\vspace{1em}
    \centering
    \setcounter{subsubfigure}{0}
    \begin{subsubfigure}[t]{0.3\textwidth}
        \centering
        \includegraphics[scale=0.5]{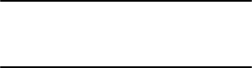}
        \caption{}
        \label{fig:thm1.1RIImove(a)}
    \end{subsubfigure} 
    \begin{subsubfigure}[t]{0.3\textwidth}
        \centering
        \includegraphics[scale=0.5]{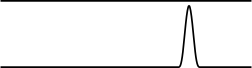}
        \caption{}
        \label{fig:thm1.1RIImove(b)}
    \end{subsubfigure} 
    \begin{subsubfigure}[t]{0.3\textwidth}
        \centering
        \includegraphics[scale=0.5]{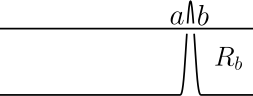}
        \caption{}
        \label{fig:thm1.1RIImove(c)}
    \end{subsubfigure} 
    \caption{Reidemeister II \\A very narrow `finger' is created in the lower strand, which crosses underneath the upper strand. The area of the bigon with endpoints $a$ and $b$, as depicted in Figure~\ref{fig:thm1.1RIImove(c)}, is arbitrarily small.}
    \label{fig:thm1.1RIImove}
\end{subfigure}

\begin{subfigure}[t]{\textwidth}\vspace{1em}
    \setcounter{subsubfigure}{0}
    \centering
    \begin{subsubfigure}[t]{0.3\textwidth}
        \centering
        \includegraphics[scale=0.5]{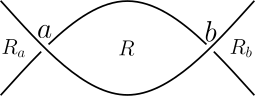}
        \caption{}
        \label{fig:thm1.1RII^{-1}move(a)}
    \end{subsubfigure} 
    \begin{subsubfigure}[t]{0.3\textwidth}
        \centering
        \includegraphics[scale=0.5]{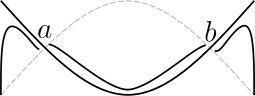}
        \caption{}
        \label{fig:thm1.1RII^{-1}move(b)}
    \end{subsubfigure} 
    \begin{subsubfigure}[t]{0.3\textwidth}
        \centering
        \includegraphics[scale=0.5]{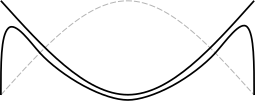}
        \caption{}
        \label{fig:thm1.1RII^{-1}move(c)}
    \end{subsubfigure} 
    \label{fig:thm1.1RII^{-1}move}
    \caption{Reidemeister II$^{-1}$\\The understrand is moved to position \ref{fig:thm1.1RII^{-1}move(b)} before performing the Reidemeister move. Once $h(b)-h(a)<\delta$, the Reidemeister move is performed ending in position \ref{fig:thm1.1RII^{-1}move(c)}.}
\end{subfigure}

\begin{subfigure}[t]{\textwidth}\vspace{1em}
    \setcounter{subsubfigure}{0}
    \centering
    \begin{subsubfigure}[t]{0.3\textwidth}
        \centering
        \includegraphics[scale=0.3]{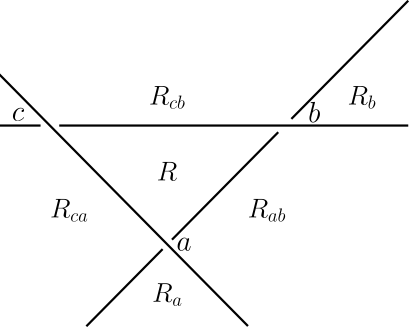}
        \caption{}
        \label{fig:thm1.1RIIIbmove(a)}
    \end{subsubfigure} 
    \begin{subsubfigure}[t]{0.3\textwidth}
        \centering
        \includegraphics[scale=0.3]{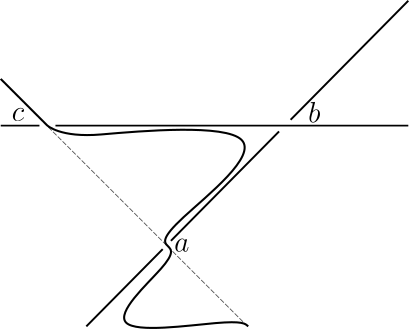}
        \caption{}
        \label{fig:thm1.1RIIIbmove(b)}
    \end{subsubfigure} 
    \begin{subsubfigure}[t]{0.3\textwidth}
        \centering
        \includegraphics[scale=0.3]{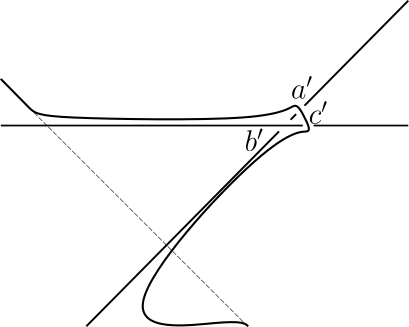}
        \caption{}
        \label{fig:thm1.1RIIIbmove(c)}
    \end{subsubfigure} 
    \caption{Reidemeister IIIb \\The strand connecting the crossings $c$ and $a$ in Figure~\ref{fig:thm1.1RIIIbmove(a)} is replaced with the new arc in Figure~\ref{fig:thm1.1RIIIbmove(c)}, which is sufficiently close to the old sides $cb$ and $ba$ so that the new heights $h(a')$ and $h(c')$ are arbitrarily close to $h(c)+h(b)$ and $h(c)$, respectively.}
    \label{fig:thm1.1RIIIbmove}
\end{subfigure}

    \caption{Reidemeister moves II, II$^{-1}$, and IIIb}
    \label{fig:enter-label}
\end{figure}

\subsection{The case of a Reidemeister IIIb move}\label{sec:R3b}

When a Reidemeister move IIIb is possible, there is a Legendrian isotopy realizing the planar isotopy from Figure~\ref{fig:thm1.1RIIIbmove(a)} to Figure~\ref{fig:thm1.1RIIIbmove(b)} by \cite{kalman2005contact}. Such an isotopy allows the area of the triangle $R$ to be made arbitrarily small; as a result $h(a)$ can be made arbitrarily close to $h(b)+h(c)$ while leaving $h(b)$ and $h(c)$ unchanged. For this reason, we assume that $h(a)>h(b)$ and $h(a)>h(c)$. When performing the Reidemeister move to go from Figure~\ref{fig:thm1.1RIIIbmove(b)} to Figure~\ref{fig:thm1.1RIIIbmove(c)}, we may satisfy the following three conditions.

\begin{itemize}
    \item The height $h(c')$ can be made arbitrarily close to $h(c)$ by Green's theorem,
    \item $h(b')=h(b)$, and
    \item $h(a')$ can be made arbitrarily close to $h(b)+h(c)$.
\end{itemize}

By \cite{kalman2005contact} we can take the area of the triangle before and after the move, respectively given by $h(b)+h(c)-h(a)$ and $h(b')+h(c')-h(a')$, to be less than $\delta/2 < 0$ and similarly take $|h(c)-h(c')| < \delta/2$. Thus we have
    \[-\delta/2<h(c')-h(c)-(h(a)-h(a'))<\delta/2.\]
    
Hence $|h(a')-h(a)|<\delta$. Thus we have a height assignment before and after the Reidemeister move IIIb such that the height of each generator changes by less than $\delta$ So by Lemma~\ref{lem:sigma-2delta-interleaving} it induces a $2\delta$-interleaving.

Note that the induced isomorphism of the Reidemeister move IIIb isotopy from $A_\Leg$ to $A_{\Leg'}$ is given by $a\mapsto a'+ \epsilon(c)b'+\epsilon(b)c'$ and $q\mapsto q'$ for $q\not=a$. But $h(a'+ \epsilon(c)b'+\epsilon(b)c')=h(a')$. Thus the isomorphism is a single semimonotonic automorphism and by Lemma~\ref{lem:semimonotonicity} this gives a $2\delta$-interleaving.
Therefore separating the Reidemeister move into two parts: a DGA isomorphism and a height reassignment, each give a $2\delta$-interleaving of chain complexes and so Proposition~\ref{prop:2delta-interleaving} we have a $2\delta$-interleaving of the barcodes.

\subsection{The case of a Reidemeister IIIa move}\label{sec:R3a}

When a Reidemeister move IIIa is possible, the change in height assignments is similar to move IIIb. The movement is the same as Figure~\ref{fig:thm1.1RIIIbmove}, except all of the crossing data is flipped. The induced isomorphism of the move IIIb isotopy is simply a relabeling $q \mapsto q'$, which is semimonotonic and thus by Lemma~\ref{lem:semimonotonicity} gives a $2\delta$-interleaving. Since the heights can change by at most $\delta$, it follows from Lemma~\ref{lem:sigma-2delta-interleaving} that the height reassignment is also $2\delta$-interleaving. Once again, by Proposition~\ref{prop:2delta-interleaving}, this induces a $2\delta$-interleaving of the barcodes that one computes before and after the Reidemeister move.

    \begin{figure}
    \label{fig:trefoilRII}
    \centering
    
    \begin{subfigure}[t]{0.4\textwidth}\vspace{1em}
        \centering
        \includegraphics[scale=0.35]{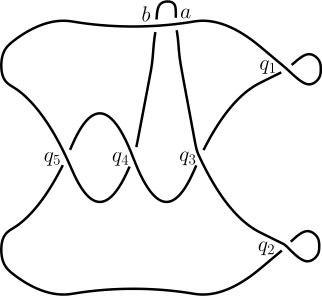}
        \caption{A Reidemeister move II has been performed on the Lagrangian diagram from Figure \ref{fig:toytrefoil(a)}. The two additional crossings are denoted by $a$ and $b$.}
        \label{fig:trefoilRII(a)}
    \end{subfigure}\hspace{0.05\textwidth}
    \begin{subfigure}[t]{0.4\textwidth}\vspace{1em}
        \centering
            \begin{tikzpicture}[xscale=1.65,yscale=2.3]
            \draw[-latex] (-.5,0) -- (3,0); \draw[dashed, lightgray] (0.5,0) -- (0.5,.5);
            \draw[dashed, lightgray] (0.5,0) -- (0.5,.8);
            \draw[dashed, lightgray] (1,0) -- (1,.8);
            \draw[dashed, lightgray] (1.3,0) -- (1.3,.8);
            \draw[dashed, lightgray] (2,0) -- (2,.9);
            \foreach \x in  {0,0.5,2}
            \node[blackpoint] at (\x,0){};
            \node[black] at (1,0){}; \node[black] at (1.3,0){};
            \node[textnode] at (0,-.2) {0};
            \node[textnode] at (0.5,-.2) {1};
            \node[textnode] at (1,-.2) {2};
            \node[textnode] at (1.3,-.2) {$2+\delta$};
            \node[textnode] at (2,-.2) {4};
            \node[textnode] at (2.9,-.2) {$h$};
            \draw[-stealth] (0.5,.1) node[blackpoint]{} node[left,scale=0.8]{$q_5$} -- (2.9,.1) {};
            \draw[-stealth] (0.5,.2) node[blackpoint]{} node[left,scale=0.8]{$q_4$} -- (2.9,.2) {};
            \draw (1,.4) node[blackpoint]{} node[left,scale=0.8]{$q_4+a$} -- (1.3,.4) node[whitepoint]{};
            \draw[-stealth] (2,.6) node[blackpoint]{} node[left,scale=0.8]{$q_1+q_2$} --  node[above]{$H_1$} (2.9,.6) {};
            \draw (0.5,.3) node[blackpoint]{} node[left,scale=0.8]{$q_3+q_5$} --  (2,.3) node[whitepoint] {};
            \draw (0.0,.05) -- ++(-.1,0) -- node[left]{$H_0$}  ++(0,.4) -- ++(.1,0);
            \end{tikzpicture}
        \caption{The barcode for the diagram in Fugure \ref{fig:trefoilRII(a)}. It also has one infinite bar in $H_1$, two infinite bars in $H_0$, and one finite bar in $H_0$ which is created at height 1 and destroyed at height 4, compared to Figure \ref{fig:toytrefoil(b)}; the only difference is an additional finite bar of length $\delta$.}
        \label{fig:trefoilRII(b)}
    \end{subfigure}    
    \caption{A diagram and its barcode for the Legendrian trefoil after a Reidemeister move II has been performed}
\end{figure}
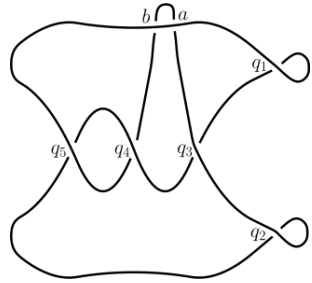
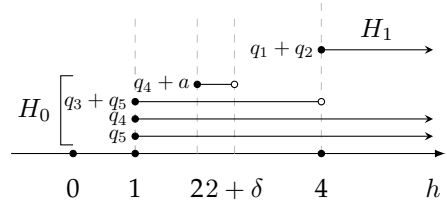

\subsection{Returning to the trefoil example}
Previously in Section~\ref{sec:trefoilex} we computed the barcode of the standard Lagrangian projection of the trefoil. Here we compute  the barcode after a Redemeister move II had been performed.  Figure~\ref{fig:toytrefoil(b)} is the barcode before this move, and Figure~\ref{fig:trefoilRII(b)} is the barcode after this move. Note that the only difference is that a bar of length $\delta$ has been added to the post-move barcode. This means that our choice of height assignment led to $2\delta$ interleaved barcodes.
{The magnitude of the parameter $\delta$ is fixed by the area of the bigon in Figure~\ref{fig:trefoilRII(a)} that has corners at $a$ and $b$.}

\section{The flooding algorithm}\label{sec:flooding}

When compared to its unfiltered analogue, persistent Legendrian contact homology has two primary shortcomings:
\begin{enumerate}
    \item Persistent LCH is not a Legendrian isotopy invariant.
    \item The computation of persistent LCH requires more data than is contained in a Lagrangian diagram.
\end{enumerate}
In particular, while a Lagrangian projection diagram allows us to compute the grading of each generator of the Chekanov-Eliashberg DGA, such a diagram does not encode the DGA's filtration\footnote{The distinction between Lagrangian projection diagrams and Lagrangian projections is important here, as one \emph{can} recover the filtration from $\Pi(\Lambda)$.}.  Of course these shortcomings are related --- while a Lagrangian projection diagram does not determine a unique Legendrian knot, it determines a unique Legendrian isotopy type.

The effect of Legendrian isotopy on persistent LCH is the focus of Theorem~\ref{thm:main-invariance-result} above.  In order to compute persistent LCH directly from a Lagrangian diagram, we might hope to find a canonical assignment of heights to the crossings of the diagram.  For a Lagrangian diagram with $n$ crossings, the set of permissible height assignments $(h_1,\ldots,h_n)$ forms an open, unbounded cone in the first orthant of $\mathbb{R}^n$, with vertex at the origin.  By declaring a minimum allowable height, as well as a minimum allowable area for the area patches of the diagram, we obtain a closed, unbounded cone disjoint from the coordinate hyperplanes of $\mathbb{R}^n$.  A natural candidate for a Lagrangian diagram's canonical height assignment is given by the (unique) vertex of this cone which is nearest the origin, and we anticipate future work investigating the nature of persistent LCH using this height assignment.

In this section, however, we take up the more modest goal of producing \emph{some} --- not necessarily canonical --- height assignment via a short computation.  While a variety of efficient algorithms exist to solve the linear programming problem described in the above paragraph, these can be labor-intensive to carry out by hand.  We describe in this section our \emph{flooding algorithm}, which can be quickly applied by a human.  The algorithm does not succeed in assigning heights to \emph{every} Lagrangian diagram, but it has the advantage of halting for all Lagrangian diagrams.  Moreover, the collection of diagrams for which the algorithm assigns heights is quite large.  For instance, Ng showed in \cite{ng2003computable} how to start with a front projection $F(\Lambda)$ and produce a Lagrangian diagram for a knot Legendrian isotopic to $\Lambda$.  The algorithm described in this section successfully assigns heights to any Lagrangian diagram obtained by applying Ng's resolution.  Thus, when one wishes to produce a barcode from a Lagrangian diagram, the flooding algorithm often provides a quick means of doing so.

We present our algorithm in subsection~\ref{sec:flooding:algorithm} and give a recursive implementation in subsection~\ref{sec:flooding:recursive}.  In subsection~\ref{sec:flooding:diagrammatic} we explain how to carry out the algorithm on a Lagrangian diagram without explicitly writing down the related area inequalities.  We then use subsection~\ref{sec:flooding:heights} to demonstrate a choice of height assignment once the algorithm terminates, and~\ref{sec:flooding:termination} to show that the algorithm succeeds in assigning heights for any Lagrangian diagram obtained via Ng's resolution.  Finally, in subsection~\ref{sec:flooding-fails-example} we give an example of a Lagrangian diagram for which the flooding algorithm fails to assign heights.

\subsection{The algorithm}\label{sec:flooding:algorithm}
Recall from subsection~\ref{sec:background:filtering} that a Lagrangian diagram determines a collection of area inequalities; each inequality ensures that a patch in the diagram has positive area. Furthermore, we use the Legendrian  condition, (\ref{eq:LegCondition}), to convert the condition of positivity of area into an inequality on the sums of heights of each crossing involved in an area patch.  Suppose that the generators of $\mathcal{A}_\Lambda$  are labeled $q_1, q_2, \cdots, q_n$, and that
\[
\mathcal{R} = \{f_1,f_2,\ldots,f_m\}
\]
is the resulting set of height inequalities for the given Lagrangian diagram, with each $f_i$ taking the form
\[
\alpha_{i,1}h(q_1) + \alpha_{i,2}h(q_2) + \cdots + \alpha_{i,n}h(q_n) > 0,
\]
for some integers $\alpha_{i,j}\in\{-2,-1,0,1,2\}$.

Note that if some particular $h(q_j)$ has non-negative coefficient $\alpha_{i,j}$ in \emph{every} inequality $f_i$, then $h(q_i)$ may be taken to be arbitrarily large, and thus each inequality in which $h(q_i)$ has a positive coefficient may be solved --- regardless of the heights assigned to other generators.  This observation allows us to ignore inequalities in which $h(q_i)$ has a positive coefficient, leading us to consider a new, smaller set of inequalities in which $h(q_i)$ does not appear.  After solving this smaller set of inequalities, we may assign $h(q_i)$ a value as large as is necessary to solve the previously ignored inequalities.

In fact, we will apply the above reduction simultaneously to all heights which appear with only non-negative coefficients, considering the corresponding generators to be our first tier of generators.  We then repeat the process, defining a second tier of generators which appear with only non-negative coefficients in our smaller collection of inequalities.  The algorithm terminates if we can successfully assign every generator $q_j$ to a tier, at which time particular height values may be chosen.

\subsection{Recursive implementation}\label{sec:flooding:recursive}

The flooding algorithm is implemented in a recursive manner as follows. We start with the collection $\mathcal{R'} = \mathcal{R}$ of inequalities, defined above.
\begin{enumerate}
    \item Let $k=1$.
    \item Set $T_k = \left\{q_j \mid \alpha_{i,j} \geq 0 \text{ in every inequality }  f_i \in \mathcal{R'} \right\}$.\label{step:declare-tier}
    \item If $T_k=\emptyset$, stop; and in this case the algorithm is said to fail.  Otherwise, for each $q_j \in T_k$, remove $f_i$ from $\mathcal{R'}$, for every $i$ where $\alpha_{i,j} >0$.\label{step:remove-inequalities}
    \item If $\mathcal{R}'\neq\emptyset$, increase $k$ by $1$ and return to Step~\ref{step:declare-tier}.  If $\mathcal{R}'=\emptyset$, define
    \[
    T_{k+1} = \{ q_i \mid q_i \notin T_\ell,  \,\text{for all}\, \ell = 1, \ldots , k \}
    \]
    and let $M=k+1$ be the total number of tiers.
\end{enumerate}
Because the set of generators $\{q_1,\ldots,q_n\}$ and the set of inequalities $\mathcal{R}$ are  both finite, this recursive process will terminate.  Provided the algorithm does not fail, the result will be a partition of the set of generators into $M$ disjoint `tiers' that we called $T_k$:
\[ T_1\sqcup \cdots \sqcup T_{M} = \{q_1,\ldots,q_n\}.\]

\subsection{Diagrammatic implementation}\label{sec:flooding:diagrammatic}

The algorithm above could be applied to any collection of linear inequalities, so let us describe the algorithm in terms of Lagrangian diagrams, motivating its name.  Each inequality in our set $\mathcal{R}$ corresponds to an area patch of our Lagrangian diagram, and our algorithm begins with none of these area patches having been `flooded' --- though we declare the unbounded exterior region of our Lagrangian diagram to be flooded.

Now recall from Section~\ref{sec:background} that each crossing $q_i$ has two quadrants with positive Reeb sign and two quadrants with negative Reeb sign.  At step $k$ of our algorithm, we add to tier $T_k$ all those crossings whose negative quadrants both lie in flooded regions, and we then flood all regions for which these crossings serve as vertices.  The flooding step corresponds to removing the inequalities from our list, and the algorithm will succeed in producing a height assignment if all area patches are eventually flooded.  See Figure~\ref{fig:island-knot} for an example of this diagrammatic implementation where the algorithm halts without assigning all crossings to tiers.

\subsection{Obtaining a valid height assignment from the flooding algorithm}\label{sec:flooding:heights}
As written, the algorithm above does not in fact assign heights to the crossings of our Lagrangian diagram, but only sorts them into tiers.  The point of these tiers is that, once heights have been assigned to all crossings in tiers $T_{k+1},\ldots,T_M$, we may satisfy the inequalities of $\mathcal{R}$ by assigning sufficiently large heights to the crossings in $T_k$.  For the sake of completeness, we now give an explicit assignment of heights in this manner.

Let $T_1,\ldots,T_M$ be the tiers produced by the flooding algorithm; we denote by $|T_i|$ the cardinality of $T_i$.  We define integers $h_1,\ldots,h_M$ recursively by declaring $h_M=1$ and then
\begin{equation}\label{eq:height-assignment}
h_k = 1+\sum_{i=k+1}^M 2h_i\,|T_i|,
\quad\text{for~}k=M-1,M-2,\ldots,1.
\end{equation}
Finally, we assign height $h_k$ to each crossing in tier $T_k$, for $1\leq k\leq M$.

\begin{proposition}
Declaring each crossing in tier $T_k$ to have height $h_k$, for $1\leq k\leq M$, gives a valid height assignment for the Lagrangian diagram.
\end{proposition}

\begin{proof} 
Consider a crossing $q_j$ in the tier $T_k$.  Because the crossings in tiers $T_1,\ldots,T_{k-1}$ will see their heights assigned after that of $q_j$, we need only concern ourselves with inequalities involving crossings from tiers $T_{k+1},\ldots,T_M$ when assigning the height $h(q_j)$.  Our construction of the tiers ensures that $h(q_j)$ appears with non-negative coefficient in any such inequality.  Indeed, consider an inequality
\[
\alpha_{i,1}h(q_1) + \alpha_{i,2}h(q_2) + \cdots + \alpha_{i,n}h(q_n) > 0
\]
with the property that $\alpha_{i,j}>0$ and $\alpha_{i,\ell}=0$ whenever $q_\ell$ is a crossing in $T_1\sqcup\cdots\sqcup T_{k-1}$.  The number of negative terms in the left hand side of this inequality is bounded above by $|T_{k+1}|+\cdots+|T_M|$, and each coefficient $\alpha_{i,\ell}$ satisfies $|\alpha_{i,\ell}|\leq 2$.  The assignment made in Equation~\ref{eq:height-assignment} then ensures that this inequality is satisfied.  In this manner we see that all inequalities involving the crossings of $T_k\sqcup\cdots\sqcup T_M$ are satisfied and, continuing inductively, that all inequalities of $\mathcal{R}$ are satisfied by our height assignment.
\end{proof}

\begin{remark}
If any of our tiers contain more than one crossing, this height assignment will be resonant, in that distinct generators of our DGA will have the same filtration value.  This can be easily overcome with a small perturbation.
\end{remark}

\subsection{Proof of termination}\label{sec:flooding:termination}
As discussed, the flooding algorithm does not always succeed in assigning heights to the crossings of a Lagrangian diagram.  In particular, if $\mathcal{R}$ is a collection of inequalities in the variables $h(q_1),\ldots,h(q_k)$ in which each variable appears at least once with a negative coefficient, then we are unable to extract a first tier of `free variables' from $\mathcal{R}$.  However, we now show that any Lagrangian diagram obtained by applying Ng's resolution to a front diagram will yield a set of area inequalities for which our algorithm terminates.

Let us fix a front diagram on which Ng's resolution will be applied and denote by $\mathcal{R}$ the collection of area inequalities corresponding to the resulting Lagrangian diagram.  In the front diagram, we enumerate the right cusps and crossings right-to-left as $c_1,\ldots,c_n$, so that
\[
x(c_1) > x(c_2) > \cdots > x(c_n).
\]
Obtaining distinct $x$-coordinates for these generators may require a small planar isotopy, but this can be done without modifying $\mathcal{R}$.

We now show by strong induction on $k$ that each generator $c_k$ will be assigned by our algorithm to some tier of generators.  Our base case is $c_1$, which must be a right cusp outside of this diagram.  Applying Ng's resolution turns this generator into a crossing $q_1$ whose negative quadrants both lie in the region outside the Lagrangian diagram, and thus are flooded when our algorithm commences.  That is, the $h(q_1)$ will not have negative coefficient in any inequality of $\mathcal{R}$, and it follows that our algorithm will assign $q_1$ to the first tier of crossings.

For our inductive step, we claim that if all generators $c_i$ are added to a tier by our algorithm, for $1\leq i<k$, then $c_k$ is added to a tier by our algorithm.  Indeed, regardless of whether $c_k$ is a right cusp or a crossing, Ng's resolution will produce a crossing in the Lagrangian diagram whose topmost and bottommost quadrants have negative Reeb signs.  Moreover, each of these quadrants will either lie outside of the Lagrangian diagram or lie in an area patch which meets a crossing $c_i$, $1\leq i<k$.  In either case, the quadrant lies in an area patch which will be flooded by our algorithm, and thus our algorithm will assign $c_k$ to a tier.

Inductively, we see that each generator $c_1,\ldots,c_n$ is assigned by our algorithm to a tier, and thus the algorithm terminates.

\subsection{An example where the flooding algorithm fails}\label{sec:flooding-fails-example}

\begin{figure}
    \centering
    \includegraphics[scale=0.85]{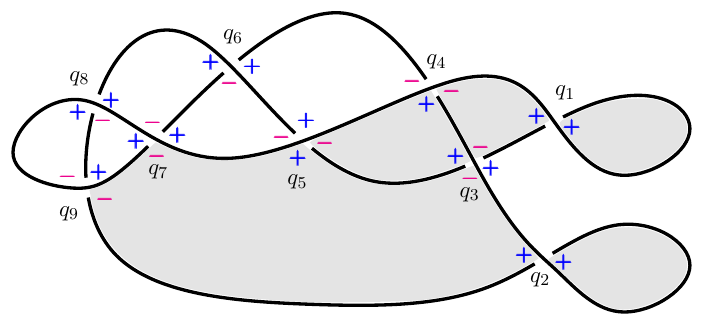}
    \caption{Example of a Lagrangian diagram for which the flooding algorithm fails to terminate. 
    }
    \label{fig:island-knot}
\end{figure}

For Lagrangian diagrams which are not obtained via Ng's resolution, the flooding algorithm is not guaranteed to succeed in assigning heights to crossings. (The algorithm still halts, however.)  Indeed, the Lagrangian diagram depicted in Figure~\ref{fig:island-knot} has area inequalities given by
\begin{align*}
h_1 &> 0 & h_3 + h_4 - h_5 &> 0 & h_9 + h_7 - h_8 &> 0 \\
h_2 &> 0 & h_5 + h_6 - h_4 &> 0 & h_8 - h_9 &> 0 \\
h_1 - h_4 - h_3 &> 0 & h_7 - h_6 - h_5 &> 0 & h_8 + h_6 - h_7 &> 0 \\
h_2 - h_3 + h_5 - h_7 - h_9 &> 0 & 
\end{align*}
Here we have written $h_i$ for $h(q_i)$.  The flooding algorithm succeeds in creating two tiers of crossings --- $T_1=\{q_1,q_2\}$ and $T_2=\{q_3\}$ --- after which our collection of inequalities is
\begin{align*}
h_5 + h_6 - h_4 &> 0 & h_8 + h_6 - h_7 &> 0 & h_8 - h_9 &> 0 \\
h_7 - h_6 - h_5 &> 0 & h_9 + h_7 - h_8 &> 0 & 
\end{align*}
Because each height appears with negative coefficient at least once, the flooding algorithm halts without successfully assigning every crossing to a tier.

This third stage of the flooding algorithm is depicted in Figure~\ref{fig:island-knot}, where shading indicates those area patches which have been flooded.  Because every crossing with unflooded quadrants has at least one \emph{negative} unflooded quadrant, we cannot flood any more area patches, and the algorithm halts.  We think of the regions which share the vertices $q_5$ and $q_6$ as `islands' which we are unable to flood, and any Lagrangian diagram with such a region will cause the flooding algorithm to fail.

\section{Strong Morse inequalities}\label{sec:morse-inequalities}

As the reader can well see, the computation of Legendrian contact homology involves associating a Morse chain complex to the critical points of the Legendrian action functional, Equation~\ref{eq:height}. In this section, we push the Morse theory further, and develop a relationship between two polynomials in Theorem~\ref{thm:strong-morse-inequality}. This relationship will be called the strong Morse inequalities in the context of persistent LCH.

The \textit{Morse-Chekanov polynomial} of a Legendrian knot $\Leg$ is defined to be  the Laurent polynomial
\begin{align}\label{MCpoly}
MC(z)=\sum_{k=-\infty}^\infty \#\{\text{Reeb chords of $\Leg$ in grading } k\} z^k.\end{align}

Similarly, given an augmentation $\epsilon:\AlgL\rightarrow \Z_2$, the \textit{Poincar\'e-Chekanov polynomial} is defined as the Laurent polynomial
\begin{align}\label{PCpoly}
PC^\epsilon(z)=\sum_{k=-\infty}^\infty\dim_{\Z_2}\left(\LCH_k^\epsilon(\Leg)\right) z^k.\end{align}

\begin{theorem}[Strong Morse Inequality]
    \label{thm:strong-morse-inequality}
    Given an augmentation $\epsilon:(\AlgL,\dL)\rightarrow (\Z_2,0)$, we have that
        \[MC(z)-PC^\epsilon(z)=(z+1)R(z)\]
    where $R(z)=\sum_k \#\{\text{Finite bars in } F^\bullet \LCH^\epsilon_k(\Leg)\} z^k.$
\end{theorem}

\begin{proof}
    By the rank-nullity theorem, we have that
        \begin{align*}
            MC(z)-PC^\epsilon(z)&=\sum_{k=-\infty}^\infty (\nul(\deps_{1,k})+\rank(\deps_{1,k})-\nul(\deps_{1,k})+\rank(\deps_{1,k+1}))z^k\\
                                &=\sum_{k=-\infty}^\infty (\rank(\deps_{1,k})+\rank(\deps_{1,k+1}))z^k\\
                                &=\sum_{k=-\infty}^\infty \rank(\deps_{1,k})z^k+\rank(\deps_{1,k})z^{k-1}\\
                                &=\sum_{k=-\infty}^\infty \rank(\deps_{1,k})(z+1)z^{k-1}\\
                                &=(z+1)\left(\sum_{k=-\infty}^\infty \rank(\deps_{1,k})z^{k-1}\right).\\
        \end{align*}
    Since $\deps$ is strictly height decreasing it follows that no homology class can be born and die simultaneously. This shows that $\rank(\deps_{1,k})$ is the number of finite bars in $F^\bullet \LCH^\epsilon_{k-1}(\Leg)$ proving the claim.
\end{proof}

From this theorem, one can see that some aspects of persistence are accessible from the standard Legendrian contact homology. However, the data coming from the actual Reeb chord heights and the generators of the finite bars in $F^\bullet \LCH^\epsilon_\ast(\Leg)$ is absent in the above result. Similar to how the normal Poincar\'e polynomial doesn't recover all the information given by (singular) homology, $R(z)$ is not enough to recover the data given by persistent Legendrian contact homology.

\begin{remark} 
    As $t\in \R$ varies, generators of $A_\Leg$ are born at $A_\Leg^t$ generating an infinite bar of $A_\Leg^\bullet$ in degree $k$ and correspondingly contributing a $z^k$ to $MC(z)$. Some of these infinite bars will correspond to infinite bars in $F^\bullet \LCH^\epsilon_k(\Leg)$ and contributing a $z^k$ to $PC^\epsilon(z)$. Left over are infinite bars in $A_\Leg^\bullet$ that correspond to killing a sum of generators of smaller height corresponding to a single finite bar in $F^\bullet \LCH^\epsilon_\ast(\Leg)$ and contributing a factor of $t^k+t^{k-1}$ to $MC(z)-PC^\epsilon(z)$. Overall, the difference of the number of generators in $A_\Leg$ and the number of generators in $\LCH^\epsilon(\Leg)$ contributes twice the number of finite bars in $F^\bullet \LCH^\epsilon(\Leg)$, and this agrees with Theorem \ref{thm:strong-morse-inequality} by setting $z=1$ and noticing that $R(1)=\frac{1}{2}(MC(1)-PC^\epsilon(1))$.
\end{remark}

\begin{example}
    We can now compute the Morse-Chekanov and the Poincar{\'e}-Chekanov polynomials for the examples given in  section~\ref{sec:toys}. We find that
        \[MC(z)-PC^\epsilon(z)=0=(z+1)\cdot 0,\]
    in the case of the unknot, and 
        \[MC(z)-PC^\epsilon(z)=z+1=(z+1)\cdot 1,\]
    in the case of the trefoil. The reader should compare these differences to the finite bars found in the  barcodes of  Figure~\ref{fig:toyunknot(b)} and Figure~\ref{fig:toytrefoil(b)}, respectively.
\end{example}

\subsection*{Acknowledgments}
This project was initiated as part of a Research Experience for Undergraduates in the School of Mathematics at Georgia Institute of Technology.  The REU was supported by Georgia Tech's College of Sciences and NSF grants \#1745583, \#1851843, \#2244427, and arranged by Igor Belegradek.  The authors are grateful to all who made this REU possible.  DI was additionally supported by  the NSF-Simons Southeast Center for Mathematics and Biology (SCMB) through the grants National Science Foundation DMS1764406 and Simons Foundation/SFARI 594594, and would like to thank Orsola Capovilla-Searle and Ziva Myer for many helpful conversations. WS was partially supported by NSF grant DMS-2104144 and would like to thank Jennifer Hom for that.

\bibliographystyle{alpha}
\bibliography{references}
\end{document}